\documentclass[11pt,a4paper]{article}

\usepackage{amssymb,amsmath,amsfonts,latexsym}
\usepackage{pstricks,pst-plot,pst-node,pst-text,pst-3d}
\usepackage{enumerate}
\usepackage{geometry,booktabs, fancyhdr}
\usepackage[english]{babel}
\usepackage{epstopdf}
\usepackage{inputenc}

\usepackage{amsmath,latexsym,mathrsfs,textcomp, fontenc,yhmath}
\usepackage{amssymb,amsfonts,amscd,amssymb,amsthm}
\usepackage{graphics,verbatim,url}
\usepackage{epic, curves}
\newcommand{\doi}[1]{\url{http://dx.doi.org/#1}}
\usepackage{index}

\newtheorem{theorem}{Theorem}[section]
\newtheorem{lemma}[theorem]{Lemma}

\newtheorem{proposition}[theorem]{Proposition}

\newtheorem{remark}[theorem]{Remark}
\newtheorem{definition}[theorem]{Definition}

\newcommand{\dcirc}{\dot{\circ}}

\newcommand{\R}{{\mathbb R}}

\newcommand{\mP}{\mathbb P}

\newcommand{\mE}{\mathbb E}

\newcommand{\cK}{{\cal K}}
\newcommand{\cW}{{\cal W}}
\newcommand{\cM}{\mathcal M}
\newcommand{\cN}{{\cal N}}
\newcommand{\cG}{{\cal G}}
\newcommand{\cT}{{\cal T}}

\newcommand{\cV}{\mathcal V}

\newcommand{\cL}{\mathcal L}

\newcommand{\inpro}[2]{\left\langle{#1},{#2}\right\rangle}
\newcommand{\inprod}[2]{\left\langle{#1},{#2}\right\rangle}

\newcommand{\sett}[1]{\left\{#1\right\}}

\newcommand{\norm}[2]{\left\|{#1}\right\|_{#2}}

\newcommand{\snorm}[2]{\left|{#1}\right|_{#2}}

\newcommand{\brac}[1]{\left(#1\right)}
\newcommand{\brap}[1]{\left[#1\right]}

\newcommand{\abs}[1]{\left|#1\right|}

\newcommand{\De}{\Delta}
\newcommand{\G}{\Gamma}

\newcommand{\Om}{\Omega}
\newcommand{\al}{\alpha}

\newcommand{\eps}{\epsilon}
\newcommand{\ga}{\gamma}
\newcommand{\om}{\omega}

\newcommand{\vecx}{\boldsymbol{x}}
\newcommand{\vecy}{\boldsymbol{y}}

\newcommand{\vecn}{\boldsymbol{n}}

\newcommand{\vecv}{\boldsymbol{v}}

\newcommand{\vx}{\boldsymbol{x}}

\newcommand{\vn}{\boldsymbol{n}}

\newcommand{\Corr}{{\rm{Cor}}}
\newcommand{\Covv}{{\rm{Covar}}}

\newcommand{\goto}{\rightarrow}
\newcommand{\wtd}{\widetilde}

\DeclareMathOperator{\supp}{{supp \/}}
\DeclareMathOperator{\divv}{{div \/}}

\DeclareMathOperator*{\esssup}{ess\,sup}

\usepackage{graphicx}

\title{A parabolic equation on domains with random 
boundaries}

\author{
Duong Thanh Pham
\thanks{Vietnamese German University, Le Lai street, Binh Duong New City, 
Binh Duong Province, Vietnam} 
\and 
Thanh Tran
\thanks{School of Mathematics and Statistics, The University of New South Wales, 
Sydney 2052, Australia, \texttt{thanh.tran@unsw.edu.au}.}
\thanks{Both authors are supported by ARC DP160101755.}
}

\date{}
\makeatletter \@addtoreset{equation}{section}
\@addtoreset{table}{section} \makeatother

\begin{document}

\maketitle

\begin{abstract}
A heat equation with uncertain domains
is thoroughly investigated. Statistical 
moments of the solution is approximated by the 
counterparts of the shape 
derivative. 
A rigorous proof for the existence of the shape derivative 
is presented. Boundary integral equation methods are used to compute 
statistical moments of the shape derivative. 
\end{abstract}

\section{Introduction}\label{s:Introduction}

Parabolic partial differential equations arise in a wide 
family of science, including heat diffusion, ocean acoustic 
propagation, physical and mathematical systems with a time 
variable, and processes having behaviour of heat diffusion 
through a solid. A typical example of parabolic partial 
differential equations is the heat equation that describes 
distribution of heat in a given region over time. Provided 
that the problem parameters are known exactly, 
nowadays powerful computers together with advanced 
numerical schemes are capable of producing highly accurate  
deterministic numerical solutions. 

However, in reality 
problem parameters are prone to uncertainty for many 
reasons. First, the parameters are often 
obtained through inexact measurements due to 
imperfect 
measurement devices. Second, the parameters are 
approximated from a large but finite number of system 
samples; this approximation can be incomplete or 
stochastic. 
Finally, mathematical problems are themselves only 
approximations of the actual processes. Under these 
circumstances, numerical results of a finite number of 
deterministic simulations with a finite set of problem 
parameters are of limited use. 
 An important
paradigm, becoming rapidly popular over the past
years, see e.g. 
\cite{BarSwbZol11,CohDVeSwb11,ChSwb13fokm,
CohDVeSwb10,
ForKor10,Git13adap,
HrbSndSwb08sm,BNobTmp07,GraKuoNuySclSlo11,SwbGit11,SwbTod06}
and the references therein, is to treat the lack of 
knowledge
via modelling uncertain parameters as random fields. 

In this paper we consider the following 
initial-boundary value problem 
\begin{subequations}\label{main1}
\begin{align}
u_t(\om) 
-\Delta u(\om) 
&
=
f
\quad 
\text{in }
Q_T(\om):= (0,T)\times U(\om),
\label{e:ut Delta u f}
\\
Bu(\om)
&
=
0
\quad 
\text{on }
\Sigma_T(\om):= (0,T)\times \partial U(\om)    
\\
u(\om)|_{t=0} 
&
=
g
\quad 
\text{in }
U(\om),
\label{e:u eps 0}
\end{align}
\end{subequations}
where $Bu(\om) = 0$ indicates 
either the Dirichlet boundary condition
\begin{equation}\label{Dir cond}
u(t,\vecx;\om) 
=
0,
\quad 
(t,\vecx)\in (0,T)\times \partial U(\om),
\end{equation}
or the Neumann boundary condition 
\begin{equation}\label{e:Neumann cond}
\frac{\partial u(t,\vecx;\om)}{\partial \vecn}
=
0,
\quad 
(t,\vecx)\in (0,T)\times\partial U(\om).
\end{equation}
Here, the domain $U(\om)$ and so does its 
boundary $\partial U(\om)$ depend on a 
``random event''
$\om\in\Om$,
where $(\Om,\Sigma,\mP)$ is a generic 
complete probability space.
In this paper, we shall estimate probabilistic 
properties of $u(\om)-\mE[u]$.
We postpone until the next section a precise 
description of 
the random domain $U^\eps(\om)$ and random 
boundary 
$\G^\eps(\om)$. 

In this article, we develop a deterministic method 
for 
numerical solution to the problems~\eqref{e:ut 
Delta u 
f}--\eqref{e:u eps 0} with either~\eqref{Dir 
cond} 
or~\eqref{e:Neumann cond}, respectively. In this 
model, the 
spatial domain on which the problem is stated 
depends on 
the ``random event'' $\om$ and the parameter 
$\eps>0$ 
controlling the amplitude of the perturbation. 
Thus, the 
solution depends on $\om$ and $\eps$ and is 
denoted by 
$u^\eps(\om)$. The case $\eps=0$ corresponds to 
the zero 
perturbation and the solution is denoted by $u^0$. 
In the 
paper, we shall estimate probabilistic properties 
of the 
solution perturbation $u^\eps(\om)- u^0$ when the 
perturbation amplitude is small.

\section{Preliminaries}\label{s:Preliminaries}
\subsection{Sobolev spaces}\label{ss:Sobolev}

In this subsection we introduce the function 
spaces 
needed for the forthcoming analysis. 
Let $U$ be a bounded domain in $\R^3$. The Sobolev 
space
$H^1(U)$ is defined, as usual, as the space of 
all 
distributions which together with their first 
order 
partial derivatives
are square integrable. 
The corresponding norm $\norm{\cdot}{H^1(U)}$ is 
defined by 
\begin{equation}\label{e:norm H1}
\norm{v}{H^1(U)}
:=
\left( 
\int_U 
\big(
\abs{v(\vecx)}^2 
+
\abs{\nabla v(\vecx)}^2 
\big)\,d\vecx
\right)^{1/2}.
\end{equation}
The space $H_0^1(U)$ is the space of all functions 
in 
$H^1(U)$ vanishing on the boundary $\partial U$ of 
$U$.
The following Friedrich--Poinca\'e inequality 
(see 
e.g.~\cite[Page 61]{TosWid05}) will be frequently 
used in 
this paper.

\begin{lemma}\label{l:Poincare inequality}
In the Sobolev space $H_0^1(U)$, 
the seminorm 
\[
\snorm{v}{H^1(U)}
:=
\left( 
\int_U 
\big(
\abs{\nabla v(\vecx)}^2 
\big)\,d\vecx
\right)^{1/2}
\]
is a norm and it is equivalent to the norm given 
by~\eqref{e:norm H1}. 
\end{lemma}
We denote by $H^{-1}(U)$ the dual space of 
$H_0^1(U)$ with 
the norm
\begin{equation}\label{e dual norm}
\norm{v}{H^{-1}(U)}
:=
\sup_{w\in H_0^1(U)\atop w\not= 0}
\frac{ 
\inpro{v}{w}
}{
\norm{w}{H^1(U)}
},
\quad 
v\in H^{-1}(U).
\end{equation}
The Sobolev space 
$H^{1/2}(\partial U)$ is defined by 
\[
H^{1/2}(\partial U)
=
\{ 
g: \partial U\goto \R\ |\ g=v \ \text{on}\ 
{\partial U}
\ (\text{in the trace sense})\
\text{for some}\
v\in H^1(U)
\}
\]
and
equipped with the following norm
\[
\norm{g}{H^{1/2}(\partial U)}
:=
\inf
\{ 
\norm{v}{H^1(U)}: v \in H^1(U) \text{ and }
g = v|_{\partial U}
\}.
\]
The dual space of $H^{1/2}(\partial U)$ 
is denoted by $H^{-1/2}(\partial U)$.

In the study of parabolic PDEs, it is important to 
identify 
functions $v: [a,b]\times U \goto \R$ which maps 
from 
$[a,b]$ into a Banach space. Let $X$ denote a 
Banach space 
with the norm $\norm{\cdot}{X}$. The space 
$L^p(a,b; X)$ is 
the space of all functions $v:[a,b]\goto X$ so 
that 
$v(t)\in 
X$ for almost all $t\in [a,b]$. The 
$L^p(a,b;X)$-norm of 
$v$ is given by 
\begin{equation*}\label{e:Lp ab X norm}
\norm{v}{L^p(a,b;X)}
=
\begin{cases}
\displaystyle
\left( 
\int_a^b 
\norm{v(t)}{X}^p\,dt
\right)^{1/p}
&
\text{ if } 
1\le p\le \infty
\\
\displaystyle
\esssup_{t\in [a,b]}
\norm{v(t)}{X}
&
\text{ if }
p = \infty.
\end{cases}
\end{equation*}
In this paper, we often work  on the space 
$L^p(a,b;X)$ 
where $p=2$. 
The space $H^1(a,b;X)$ is a subspace of 
$L^2(a,b;X)$ 
consisting all functions $v:[a,b]\times X\goto 
\R$ 
satisfying $\partial v/\partial t\in L^2(a,b;X)$.
The corresponding norm is defined by 
\begin{equation}\label{e:H1 ab X norm}
\norm{v}{H^1(a,b;X)}
=
\left(
\int_a^b 
\left[
\norm{v(t)}{X}^2 
+
\norm{\frac{\partial v(t)}{\partial t}}{X}^2
\right]\,dt
\right)^{1/2}.
\end{equation}
The space $C([a,b];X)$ consists of continuous 
function 
$v:[a,b]\goto X$. The $C([a,b];X)$-norm is given 
by 
\begin{equation}\label{e:C ab X norm}
\norm{v}{C([a,b];X)}
=
\max_{t\in[a,b]}
\norm{v(t)}{X}.
\end{equation}
We note here that the spaces $L^p(a,b;X)$, 
$H^1(a,b;X)$ and 
$C([a,b];X)$ are Banach spaces for all $p\ge 1$.

\subsection{Bochner spaces}\label{ss:Bochner}
Throughout this paper we denote by $(\Om, 
\Sigma,\mP)$ a 
generic complete probability space. Let $X$ be a 
Banach 
space.
For any $1\leq k \leq \infty$, the Bochner space 
$\cL^k(\Om, 
X)$ is defined as usual by
\begin{equation}\label{Lk-def}
\cL^k(\Omega,X): = 
\big\{v: \Omega\goto X, \text{ measurable}: 
\|v\|_{\cL^k(\Omega,X)} <\infty \big\}
\end{equation}
with the norm
\begin{equation}\label{Lk-norm-def}
\norm{v}{\cL^k(\Omega, X)} :=
\left\{
\begin{array}{cr} \displaystyle
\left(\int_{\Omega}\norm{v(\omega)}{X}^k d\mathbb
P(\omega)\right)^{1/k}, & 1\leq k < \infty,\\[3ex]
\mathop{\rm ess

sup}\limits_{\omega\in 
\Omega}\norm{v(\omega)}{X}, & k=\infty.
\end{array}
\right.
\end{equation}
The elements of $\cL^k(\Omega,X)$ are called 
\emph{random 
fields}. 
We remark that for a part of the subsequent 
analysis we may 
restrict to the special case when $X$ is a 
Hilbert 
space.
In particular, 
when $X_1$ and $X_2$ are two separable Hilbert 
spaces, 
their 
tensor 
product $X_1 \otimes X_2$ is a separable Hilbert 
space with 
the natural 
inner product extended by linearity from $\langle 
v\otimes 
a, 
w\otimes b\rangle_{X_1 \otimes X_2} = \langle v, 
w\rangle_{X_1}
\langle a, b\rangle_{X_2}$, cf. e.g. \cite[p. 
20]{LightCheney80}, 
\cite[Definition 12.3.2, p.298]{Aubin00}. In this 
{paper} 
we 
work 
with $k$-fold tensor products of Hilbert spaces
\begin{equation}\label{Xk-def}
 X^{(k)} := X \otimes \dots \otimes X.
\end{equation}
with the natural inner product satisfying 
$\langle 
v_1\otimes \dots\otimes v_k, w_1\otimes 
\dots\otimes 
w_k\rangle_{X^{(k)}} = \langle v_1, 
w_1\rangle_{X} 
\dots 
\langle v_k, w_k\rangle_{X}$.

\begin{definition}
For a random field $v\in \cL^k(\Omega, X)$,
its $k$-order moment $\cM^k[v]$ is an element of 
$X^{(k)}$ 
defined by
\begin{equation}\label{moments-def}
\cM^k [v] :=  \int_{\Omega} 
\big(
\underbrace{v(\omega)\otimes\cdots\otimes 
v(\omega)}_{k\textrm{-times}}
\big)
\,d\mathbb P(\omega).
\end{equation}
\end{definition}

In the case $k=1$, the statistical moment $\cM^1 
[v]$ 
coincides with the \textit{mean value}
of $v$ and is denoted by $\mE [v]$. If $k\geq 2$, 
the statistical moment $\cM^k [v]$ is the 
\textit{$k$-point autocorrelation function} of 
$v$.
The quantity $\cM^k [v - \mE[v]]$ is termed the 
$k$-th 
central moment of $v$. We distinguish in 
particular 
second order moments: the \textit{correlation} 
and 
\textit{covariance} 
defined by
\begin{equation}
\Corr[v] := \cM^2[v] \quad\text{and}\quad  
\Covv[v] := 
\cM^2[v - \mE[v]].
\end{equation}
In this paper we work with $X$ being Sobolev 
spaces of 
real-valued functions defined on a domain 
${U\subseteq\R^3}$
yielding, in  particular, the representation
\begin{equation}\label{Corr-def}
{\Corr[v](\vecx,\vecy)}
:=
\int_{\Om}
v(\vecx,\om) 
v(\vecy,\om)
\,d\mP(\om),
\quad
\vecx,\vecy\in U.
\end{equation}
We observe that $\Corr[v]$ is defined on the 
Cartesian 
product $U \times U$. 
Similarly, $\cM^k[v]$ is defined on the $k$-fold 
Cartesian 
product $U \times \dots \times U$. {Here,} 
the dimension of the underlying domain grows 
rapidly with 
increasing moment order $k$.

%
%
%

\subsection{Random domains}

In this subsection, we describe the random 
domain and its boundary on which the 
initial-boundary value problem~\eqref{main1} is 
stated.
Let $U^0$ be a fixed bounded domain in $\R^n$, 
$n=2,3$. 
Then 
the boundary 
$\G^0:=\partial U^0$ is a closed manifold in 
$\R^n$. We 
assume that $\G^0\in 
C^{1,1}$ so that the outward normal vector 
$\vecn^0$ to 
$\G^0$ satisfies 
$\vecn^0\in C^{0,1}(\G^0)$. Suppose that 
$\kappa\in 
\cL^k(\Om, 
C^{0,1}(\G^0)$ is 
a random field, i.e. for almost any realization 
$\om\in 
\Om$, 
we have $\kappa(\cdot, \om)\in C^{0,1}(\G^0)$. 
For 
some 
sufficiently small, 
nonnegative $\eps$,  we consider a family of 
random 
closed 
surfaces of the form 
\begin{equation}\label{e:G eps def}
\G^\eps(\om) 
=
\{
\vecx + \eps \kappa(\vecx,\om)\vecn^0(\vecx) : 
\vecx\in \G^0
\},
\quad 
\om\in\Om.
\end{equation}
The bounded domain which is surrounded by 
$\G^\eps(\om)$  is denoted by $U^\eps(\om)$.
Here, the uncertainty is represented by the 
uncertainty in 
$\kappa(\cdot, \om)$.
We assume further that the random perturbation 
amplitude 
$\kappa(\vecx,\om)$ is 
centered, i.e., 
\begin{equation}\label{e:kappa 0}
\mE[\kappa(\vecx,\cdot)] 
=
0
\quad 
\forall \vecx\in \G^0,
\end{equation}
and $\kappa$ is uniformly bounded, i.e., there 
exist 
bounded 
domain $U_-$ and 
$U_+$ satisfying 
\begin{equation}\label{D D eps}
U_-\subset U^\eps(\om)\subset U_+
\quad 
\forall 
\om\in \Om,
\quad 
\forall \eps\le \eps_0,
\end{equation}
for some sufficiently small and positive $\eps_0$.
Due to~\eqref{e:kappa 0}, the mean random 
boundary 
satisfies 
\[
\mE[\G^\eps]
=
\{ 
\vecx + \eps \mE[\kappa]\vecn^0(\vecx) : \vecx\in 
\G^0
\}
=
\G^0
\]
and $\Covv[\kappa] = \Corr[\kappa]$.
We consider initial-boundary value problem on 
random 
domains $U^\eps(\om)$,
\begin{equation}\label{e:initial rand prob}
\begin{aligned}
u_t^\eps(\om) - \Delta u^\eps(\om) 
&
= f
\quad 
\text{in } Q_\cT^{\eps}(\om):= (0,T)\times 
U^\eps(\om)
\\
B u^\eps(\om) 
&
=
0
\quad 
\text{on }
\sigma_\cT^{\eps}(\om):= (0,T)\times \G^\eps(\om)
\\
u^\eps(\om)|_{t=0}
&
=
g
\quad 
\text{in }
U^\eps(\om).
\end{aligned}
\end{equation}
The randomnesses in the domain $U^\eps(\om)$ and 
its boundary result in randomness of the 
solution $u^\eps(\cdot,\om)$. Here, the solution 
operator $u^\eps(\om) = {\rm Sol}(U^\eps(\om))$ 
is nonlinear.
Thus, linearisation by using shape 
calculus is in demand. In this process, existence 
of a 
shape derivative of the solutions of 
deterministic 
perturbed problems has to be clarified. The shape 
derivative 
will then be  used to approximate statistical 
moments of 
the solution.

\section{Shape calculus}

\subsection{Deterministic perturbed domains}
In this section, we aim to prove the existence of 
shape 
derivative of the solution $u^\eps$ 
of~\eqref{e:initial rand prob}, which will then 
be 
used in linearisation development of the solution 
$u^\eps$ 
with respect to the perturbed domain $U^\eps$. In 
this 
section, we temporarily stay away from randomness 
and only 
work on 
deterministic perturbed domains.
Let $U^0$ be a fixed bounded domain in $\R^n$, 
$n=2,3$. 
Assume that  
the boundary 
$\G^0:=\partial U^0$ is a closed manifold in 
$\R^n$ satisfying  $\G^0\in 
C^{1,1}$.
Let
$\kappa\in 
C^{0,1}(\G^0)$.  For 
any $\eps\in (0,\eps_0)$, where $\eps_0$ is some 
sufficiently small positive number, we consider a 
family of 
deterministic closed 
surfaces of the form 
\begin{equation}\label{e:G eps def}
\G^\eps
=
\{
\vecx + \eps V(\vecx) : 
\vecx\in \G^0
\}.
\end{equation}
The bounded domain surrounded by $\G^\eps$ is 
denoted by 
$U^\eps$.
Analogously to~\eqref{D D eps}, we assume that 
$V$ is uniformly bounded, i.e., there exist 
bounded 
domain $U_-$ and 
$U_+$ satisfying 
\begin{equation}\label{D D eps1}
U_-\Subset U^\eps\Subset U_+
\quad 
\forall 
\eps\le \eps_0,
\end{equation}
where $U^\eps\Subset U_+$ means that the closures 
of all 
$U^\eps$ are proper subsets of $U_+$ for all 
$\eps\le 
\eps_0$.
Following~\cite{SokZol92},
we define a mapping $\cT^\eps:U_+\goto U_+$ which 
transforms 
$\G^0$ into $\G^\eps$ and $U^0$ into $U^\eps$, 
respectively, by 
\begin{equation}\label{e:T eps define}
\cT^\eps(\vecx)
:=
\vecx
+
\eps \tilde{V}(\vecx),
\quad
\vecx\in U_+,
\end{equation}
where $\tilde{V}$ is 
any 
smoothness-preserving extensions of $V$. Without 
loss of generality, we may 
assume that 
\[
\supp(\tilde{V}):=
\overline{\{\vecx\in U_+: \wtd V(\vecx)\not= 
0\}}
\]
is a proper subset of $U_+$.
Denoting by 
\begin{equation}\label{e:U star}
U^*:= \supp(\tilde{V}), 
\end{equation}
the set $U^*$ is a compact subset in 
$U_+$. In this paper,
we require in particular that $\tilde{V}\in 
W^{1,\infty}(U_+)$. 
For the ease of notation we also use  
$V$ for its extension in the rest of the paper.
In~\cite{SokZol92}, $V$ is called \textit{the 
velocity 
field} of the mapping $\cT^\eps$. 
In the present paper, for any function $v$ 
defined 
on the 
$[0,T]\times U^\eps$, we denote 
\[
v\dot{\circ} \cT^\eps(t,\vecx)
:=
v(t, \cT^\eps(\vecx)),
\quad 
(t,\vecx)\in [0,T]\times U^0
\]
for notational convenience.

In the subsequent analysis, for any $3$ 
by $3$ matrix  $M(\vecx)$ 
whose entries are functionals of $\vecx\in 
U_+\subset\R^3$,
we denote
\[
\norm{M(\cdot)}{L^p(U)}
:=
\max_{i,j=1,2,3}
\{ 
\norm{M_{i,j}(\cdot)}{L^p(U)}
\},
\quad
{1\le p\le \infty},
\]
where $M_{ij}$ are components of $M$.
In this section, we assume that $\cT^\eps$ is 
defined 
by~\eqref{e:T eps define} where 
$\tilde{V}
\in W^{1,\infty}(U_+)$, 
and denote its Jacobian matrix and Jacobian 
determinant by
$J_{\cT^\eps}(\cdot)$ and 
$\gamma(\eps,\cdot)$, 
respectively.
It can be prove that a function $v$ belongs to 
$H^1(U^\eps)$ ($H_0^1(U^\eps)$ or $L^2(U^\eps)$, 
resp.) if 
and only if $v{\circ}\cT^{\eps}$ belongs to 
$H^1(U^0)$ 
($H_0^1(U^0)$ or $L^2(U^0)$, resp.) and there hold
\begin{subequations}\label{e norm T eps}
\begin{align}
\norm{v{\circ}\cT^\eps}{H^1(U^0)}
&
\simeq 
\norm{v}{H^1(U^\eps)},
\quad 
v\in H^1(U^\eps),
\label{9M1f1}
\\
\norm{v{\circ}\cT^\eps}{L^2(U^0)}
&
\simeq 
\norm{v}{L^2(U^\eps)},
\quad 
v\in L^2(U^\eps).
\label{9M1f2}
\end{align}
\end{subequations}

The following lemmas which will be frequently 
used 
in the 
rest of this section state some important 
properties of the 
transformation $\cT^\eps$.

\begin{lemma}\label{l:9M1a}
Assume that $V\in W^{1,\infty}(U_+)$. 
The Jacobian determinant $\ga(\eps,\cdot)$ of the 
tranformation $\cT^{\eps}$ satisfies
\begin{align}
\lim_{\eps\goto0}
\norm{\ga(\eps,\cdot) -1}{L^\infty(U_+)}
=
0,
\label{9M1h}
\end{align}
and 
\begin{align}
\lim_{\eps\goto0}
\norm{\frac{\ga(\eps,\cdot) 
-1}{\eps} - \divv V}{L^\infty(U_+)}
=
0.
\label{9Oct2c}
\end{align}

\end{lemma}

\begin{proof}
Recalling~\eqref{e:T eps define},
we
denote $V(\vecx):= 
(V_1(\vecx), V_2(\vecx), V_3(\vecx))^\top$.
The Jacobian 
matrix and the Jacobian determinant
of
$\cT^\eps$ are given by
\begin{equation}\label{e:Jc T}
J_{\cT^{\eps}}(\vecx)
=
\begin{bmatrix}
1 + \eps\dfrac{\partial V_1(\vecx)}{\partial x_1}
&
\eps\dfrac{\partial V_1(\vecx)}{\partial x_2}
&
\eps\dfrac{\partial V_1(\vecx)}{\partial x_3}
\\
\eps\dfrac{\partial V_2(\vecx)}{\partial x_1}
&
1+\eps\dfrac{\partial V_2(\vecx)}{\partial x_2}
&
\eps\dfrac{\partial V_2(\vecx)}{\partial x_3}
\\
\eps\dfrac{\partial V_3(\vecx)}{\partial x_1}
&
\eps\dfrac{\partial V_3(\vecx)}{\partial x_2}
&
1+\eps\dfrac{\partial V_3(\vecx)}{\partial x_3}
\end{bmatrix}
\end{equation}
and 
\begin{align}
\gamma(\eps,\vecx)
&
=
{\Big{|}}
1 + \eps
\Big(
\sum_{k = 1}^3
\dfrac{\partial V_k(\vecx)}{\partial x_k}
\Big)
+
\eps^2
\Big(
\sum_{k,l = 1\atop k\not= l}^3
\dfrac{\partial V_k(\vecx)}{\partial x_k}
\dfrac{\partial V_l(\vecx)}{\partial x_l}
-
\dfrac{\partial V_l(\vecx)}{\partial x_k}
\dfrac{\partial V_k(\vecx)}{\partial x_l}
\Big)
\notag
\\
&
\quad +
\eps^3
\Big(
\sum_{i,j,k = 1}^3
{\rm sign}{(i,j,k)}
\dfrac{\partial V_i(\vecx)}{\partial x_1}
\dfrac{\partial V_j(\vecx)}{\partial x_2}
\dfrac{\partial V_k(\vecx)}{\partial x_3}
\Big)
{\Big{|}}
\notag
\\
&
=:
{\big{|}}
1 + \eps\gamma_1(\vecx) + \eps^2\gamma_2(\vecx) + 
\eps^3\gamma_3(\vecx)
{\big{|}}.
\label{e:tt 1}
\end{align}
Here ${\rm sign}(i,j,k)$ denotes the sign of the 
permutation $(i,j,k)$. The entries 
$A_{ij}(\eps,\vecx)$, 
$i,j = 1,2,3$, 
of the matrix $A(\eps,\vecx)$
are given by
\begin{equation}\label{e:A eps comp}
A_{ij}(\eps,\vecx)
=
\gamma(\eps,\vecx)^{-1}
\left(
\delta_{ij}
+
\sum_{n=1}^4 \eps^n h_{ijn}(\vecx)
\right),
\end{equation}
where $h_{ijn}$ is a polynomial of 
partial derivatives of $V$ and 
$\delta_{ij}$ is the Kronecker delta. 
Since $V \in W^{1,\infty}(U_+)$, we 
deduce 
\begin{equation}\label{e:st 4}
\begin{gathered}
\gamma_n,\ h_{ijn}\in L^\infty(U_+)\cap L^2(U_+),
\quad
i,j =1,2,3
\
\text{and}
\
n=1,\ldots,4,
\\
\lim_{\eps\goto 0}
\norm{\gamma(\eps,\cdot)}{L^\infty(U_+)}
>0,
\end{gathered}
\end{equation}
where $\gamma_1$, $\gamma_2$, $\gamma_3$ are 
defined 
by~\eqref{e:tt 1} and 
$\gamma_4:= 0$ for notational convenience later.
In particular, for sufficiently small $\eps>0$, 
there holds
\begin{equation}\label{e:vbn 4}
\gamma(\eps,\vecx)
=
1 + \eps\gamma_1(\vecx) + \eps^2\gamma_2(\vecx) + 
\eps^3\gamma_3(\vecx) 
\geq c > 0
\qquad
\forall \vecx\in U_+.
\end{equation}
We then have 
\begin{align}
\lim_{\eps\goto \infty}
\norm{\ga(\eps,\cdot) - 1}{L^\infty(U_+)}
&
=
\lim_{\eps\goto \infty}
\eps 
\norm{\ga_1 + \eps \ga_2 + \eps^2 
\ga_3}{L^\infty(U_+)}
=
0,
\notag
\end{align}
noting~\eqref{e:st 4}. 
Furthermore, it follows from~\eqref{e:vbn 4} 
and~\eqref{e:tt 1} that 
\begin{align}
\lim_{\eps\goto0}
\norm{\frac{\ga(\eps,\cdot) 
-1}{\eps} - \divv V}{L^\infty(U_+)}
&
=
\eps\norm{\gamma_2 + \eps\gamma_3}{L^\infty(U_+)}.
\notag
\end{align}
Letting $\eps$ go to zero and noting~\eqref{e:st 
4}, we obtain~\eqref{9Oct2c}, completing the 
proof 
of the lemma.

\end{proof}

\begin{lemma}\label{l:T eps prop} 
Assume that $V\in W^{1,\infty}(U_+)$.
{Consider} $A(\eps,\cdot):= \gamma(\eps,\cdot) 
J_{\cT^\eps}^{-1} J_{\cT^\eps}^{-\top}$,
where $J_{\cT^\eps}^{\top}$ {is} the transpose of 
$J_{\cT^{\eps}}$. {Then} there hold
\begin{equation}\label{e:st 1}
\lim_{\eps\goto 0} \norm{A(\eps,\cdot) - 
I}{L^\infty(U_+)} 
= 0
\end{equation}
and
\begin{equation}\label{e:st 2}
\lim_{\eps\goto 0} 
\norm{\dfrac{A(\eps,\cdot) - I}{\eps} - 
A'(0,\cdot)}{L^\infty(U_+)} = 0.
\end{equation}
Here, $A'(0,\cdot)$ is the G\^ateaux derivative 
of 
$A(\eps,\cdot)$ 
at $\eps =0$, namely
\[
A'(0,\vecx)
=
\lim_{\eps\goto 0}
\frac{A(\eps,\vecx) - I(\vecx)}{\eps},
\quad\vecx\in U_+.
\]
\end{lemma}

\begin{proof}
The proof of the lemma can be done in the same 
manner as
the proof of~\cite[Lemma 3.1]{ChePhaTra15}.
\end{proof}

\begin{lemma}\label{l:9M1b}
For any function $v\in L^2([0,T]\times U_+)$, 
there holds
\begin{equation}\label{9M1m}
\lim_{\eps\goto 0} 
\norm{
v\dot{\circ}\cT^\eps
\ga(\eps,\cdot) - v}{L^2(0,T;L^2(U_+))}
= 
0.
\end{equation}
\end{lemma}

\begin{proof}
We then have
\begin{align}
\int_0^T
\norm{
\big(\ga(\eps,\cdot)
-1\big)v\dcirc\cT^\eps(\tau)
}{L^2(U^0)}^2\,d\tau
 &
\le
\norm{\gamma(\eps,\cdot) -1}{L^\infty(U^0)}^2\,
\int_0^T
\norm{v\dcirc\cT^\eps(\tau)}{L^2(U^0)}^2d\tau
\notag
\\
&
\leq
C\eps^2\,
\int_0^T
\norm{v\dcirc\cT^\eps(\tau)}{L^2(U^0)}^2d\tau.
\label{e:st 10}
\end{align}
Using the change of variables $\vecy 
=\cT^\eps(\vecx)$ and 
noting~\eqref{e:st 4}, we have 
\begin{align}
\norm{v\dcirc\cT^\eps(\tau)}{L^2(U^0)}^2
&
=
\int_{U^\eps}
\snorm{v(\tau,\vecy)}{}^2
\big(\gamma(\eps, 
(\cT^{\eps})^{-1}(\vecy))\big)^{-1}
\,d\vecy
\le 
C
\norm{v(\tau)}{L^2(U^\eps)}^2
\quad 
\forall \tau\in (0,T).
\notag
\end{align}
Here, the constant $C$ is independent of $\tau$. 
Thus,
\[ 
\int_0^T
\norm{v\dcirc\cT^\eps(\tau)}{L^2(U^0)}^2d\tau
\lesssim 
\int_0^T
\norm{v(\tau)}{L^2(U^\eps)}^2d\tau
\le 
\int_0^T
\norm{v(\tau)}{L^2(U_+)}^2d\tau.
\]
This together with~\eqref{e:st 10} implies 
\begin{equation*}\label{ea1}
\int_0^T
\norm{
\big(\ga(\eps,\cdot)
-1\big)v\dcirc\cT^\eps(\tau)
}{L^2(U^0)}^2\,d\tau
\lesssim 
\eps^2 
\int_0^T
\norm{v(\tau)}{L^2(U_+)}^2d\tau.
\end{equation*}
Therefore,
\begin{equation}\label{e:st 9}
\lim_{\eps\goto 0}
\int_0^T
\norm{\big(\ga(\eps,\vecx)
-1\big)v\dcirc\cT^\eps(\tau)}{L^2(U^0)}^2d\tau
=
0.
\end{equation}

Assume that 
$v$ belongs to  $C([0,T]\times U_+)$. Then $v$ is 
continuous and thus uniformly continuous on the 
compact set $[0,T]\times U^*$. Furthermore, since 
$\wtd \kappa\in W^{1,\infty}(U_+)$, there holds 
\[
\lim_{\eps\goto 0}
\norm{\cT^\eps(\vecx)-\vecx}{L^\infty(U^*)}
=
0.
\]
Thus,
\[
\lim_{\eps\goto 0}
\norm{v\dcirc \cT^\eps - 
v}{L^\infty([0,T]\times U^*)} 
=
0.
\]
Noting~\eqref{e:U star}, the difference $v\dcirc 
\cT^\eps - v$ vanishes outside $U^*$ and thus 
 \[
\lim_{\eps\goto 0}
\int_0^T 
\norm{ 
(v\dcirc \cT^\eps-v)(\tau)
}{L^2(U_+)}^2d\tau 
= 
\lim_{\eps\goto 0}
\int_0^T 
\norm{ 
(v\dcirc \cT^\eps-v)(\tau)
}{L^2(U^*)}^2d\tau
=
0.
\]
By using a density argument we 
deduce that 
\begin{align}
\lim_{\eps\goto 0}
\int_0^T 
\norm{ 
(v\dcirc \cT^\eps-v)(\tau)
}{L^2(U_+)}^2d\tau = 0  
\quad
\forall
v\in L^2([0,T]\times U_+).
\notag
\end{align}
This together with~\eqref{e:st 9} (by using the 
{triangle} inequality) yields the desired 
equality.
\end{proof}

\begin{lemma}\label{lem:T eps f 2}
For any function $v\in
H^1(U_+)$, there hold
\begin{align}
\lim_{\eps\goto 0}
\norm{\frac{v{\circ}\cT^\eps - v}{\eps} 
-
V\cdot \nabla  v
}{L^2(U_+)}
&
=
0,
\label{e:T g 2}
\\
\lim_{\eps\goto 0}
\norm{
\frac{\ga(\eps,\cdot) (v{\circ}\cT^\eps) - 
v}{\eps}
-
\divv\big{(} vV \big{)}
}{L^2(U_+)}
&
=
0.
\label{e:T g 1}
\end{align}

\end{lemma}

\begin{proof}
First of all, the equality~\eqref{e:T g 2} can be 
proved by 
using the density argument in which we shall 
prove~\eqref{e:T g 2}
for an arbitrary function $v\in C^\infty(U_+)$ 
and 
because of
the density of $C^\infty(U_+)$ in $H^1(U_+)$, the 
equality 
is also true for all functions in 
$H^1(U_+)$.
Indeed, let $v\in C^\infty(U_+)$.
Applying the mean
value theorem, for any $\vecx\in U^*$, there 
exists a 
$\theta_{\vecx}\in (0,1)$ such that 
\begin{equation*}
v(\cT^\eps(\vecx)) - v(\vecx)
=
(\cT^\eps(\vecx) - \vecx)\cdot
\nabla v\big(\theta_{\vecx} \cT^\eps(\vecx) + 
(1-\theta_{\vecx})\vecx\big).
\end{equation*}
This gives, noting~\eqref{e:T eps define},  
\begin{align}
\frac{v(\cT^\eps(\vecx)) - v(\vecx)
}{\eps}
-
V(\vecx)\cdot \nabla v(\vecx)
&
=
V(\vecx)\cdot 
\nabla
\left[ 
v\big(\theta_{\vecx} \cT^\eps(\vecx) + 
(1-\theta_{\vecx})\vecx\big)
-
v(\vecx)
\right].
\label{e:vi1}
\end{align}
Since $\wtd\kappa\in W^{1,\infty}(U^*)$, 
$\lim_{\eps\goto 0} 
\norm{\cT^\eps(\vecx)-\vecx}{L^\infty(U^*)} = 0$. 
If $v\in 
C^\infty(U^*)$, its partial derivatives are 
uniformly 
continuous in $U^*$. Thus 
\begin{align}
\lim_{\eps\goto 0}
\norm{ 
\nabla 
\left[ 
v\big(\theta_{\vecx} \cT^\eps(\vecx) + 
(1-\theta_{\vecx})\vecx\big)
-
v(\vecx)
\right]
}{L^\infty(U^*)}
=
0.
\notag
\end{align}
This together with~\eqref{e:vi1} implies
\begin{align}
\lim_{\eps\goto 0}
\norm{\frac{v{\circ}\cT^\eps - v}{\eps} 
-
V\cdot \nabla  v
}{L^2(U_+)}
&
=
\lim_{\eps\goto 0}
\norm{\frac{v{\circ}\cT^\eps - v}{\eps} 
-
V\cdot \nabla  v
}{L^2(U^*)}
=
0.
\label{9Au1e}
\end{align}
We next apply 
the triangle
inequality to  obtain
\begin{align}
\norm{
\frac{\ga(\eps,\cdot) (v{\circ}\cT^\eps) -
v}{\eps}
-
\divv\big{(} vV \big{)}
}{L^2(U_+)}
&
\le
\norm{\frac{\gamma(\eps,\cdot)-1}{\eps}(v{\circ}
\cT^\eps) 
-
v\divv V}{L^2({U_+})}
\notag
\\
&
+
\norm{\frac{v{\circ}\cT^\eps - v}{\eps} 
-
V\cdot \nabla  v
}{L^2({U_+})}.
\label{9Au1f}
\end{align}
Recall from~\eqref{e:tt 1} that $\gamma_1=\divv 
V$. 
It follows from~\eqref{e:vbn 4} that
\begin{align}
\frac{\gamma(\eps,\cdot) - 
1}{\eps}(v\circ\cT^\eps)
-
v \divv V
&
=
\gamma_1
(v\circ\cT^\eps -v)
+
\eps(\gamma_2 + \eps\gamma_3)
(v\circ\cT^\eps).
\notag
\end{align}
Employing the density argument as in proof of 
Lemma~\ref{l:9M1b}
and noting~\eqref{e:st 4}, 
we obtain
\begin{equation*}
\lim_{\eps\goto 0}
\norm{\gamma_1(v\circ\cT^\eps -v)
}{L^2({U_+})}
=
0.
\end{equation*}
Noting~\eqref{e:st 4}, we deduce 
\[
\lim_{\eps\goto 0}
\norm{\eps(\gamma_2 + \eps\gamma_3)(v\circ 
\cT^{\eps})}{L^2({U_+})}
=
0.
\]
Hence, 
\begin{align}
\lim_{\eps\goto 0}
\norm{\frac{\gamma(\eps,\cdot)-1}{\eps}
(v\circ\cT^\eps) 
-
v\divv V}{L^2({U_+})}
=
0.
\label{9Au1g}
\end{align}
The equality~\eqref{e:T g 2} can be derived 
from~\eqref{9Au1e}--\eqref{9Au1g}. 
This completes the proof of the lemma.

\end{proof}

\begin{lemma}\label{l:f T eps}
Let $v\in L^2(0,T; H^2(U_+))\cap 
H^1(0,T;L_2(U_+))$. There hold
\begin{align}
\lim_{\eps\goto 0}
\norm{ 
\frac{
v\dcirc\cT^\eps - v
}{\eps}
-
V\cdot \nabla v
}{L^2(0,T;L^2(U_+))}
& 
=
0,
\label{e:fT1}
\end{align}
and
\begin{align}
\lim_{\eps\goto 0}
\norm{ 
\frac{\gamma(\eps,\cdot)-1}{\eps}
v\dcirc\cT^\eps 
-
v\divv V
}{L^2(0,T;L^2(U_+))}
&
=
0.
\label{e:fT2}
\end{align}

\end{lemma}

\begin{proof}
This proof can be obtained by employing similar 
arguments as used in the proof 
of Lemma~\ref{lem:T eps f 2}, noting that $v\in 
C(0,T; H^1(U_+))$ 
(\cite[Theorem~4]{Evans10}).

\end{proof}

\subsection{Material and shape 
derivatives}\label{subsec:mat 
shape der}

\begin{definition}\label{def:Mat-ShapeDeriv}
For any sufficiently small $\eps$, let $v^\eps$ 
be 
an 
element in 
$H^1(U^\eps)$. 
The material derivative of $v^\eps$, denoted by 
$\dot v$, 
is defined by
\begin{equation}\label{MaterialDeriv-def}
\dot v
:=
\lim_{\eps\goto 0}
\frac{v^\eps\circ\cT^\eps - v^0}{\eps},
\end{equation}
if the limit exists in the corresponding space 
$H^1(U^0)$.
The \textit{shape derivative} of $v^\eps$ is 
defined 
by 
\begin{equation}\label{ShapeDeriv-def}
v'
=
\dot v - \nabla v^0\cdot V.
\end{equation}
\end{definition}

\begin{lemma}\label{lem:v shap mat K}
If $v'$ is a shape derivative of $v^\eps\in 
H^1(U^\eps)$, 
then for 
any compact set $K\subset\subset U^0$ we have 
\begin{equation}\label{ShapeDeriv-limit}
v'
=
\lim_{\eps\goto 0}
\frac{v^\eps-v^0}{\eps}\quad
\text{in}
\quad
H^1(K).
\end{equation}
\end{lemma}

\begin{proof}
Given $K\subset\subset U^0$, 
there exists an $\eps_0>0$ such that 
$K\subset\subset 
U^\eps$ for all 
$0\le \eps\le \eps_0$.
We denote by $\mathcal T: 
[0,\eps_0]\times\R^3\goto\R^3$ 
the mapping given by
\[
\mathcal T(\eps,\vecx)
:=
\cT^{\eps}(\vecx),
\quad
\forall 
(\eps,\vecx)
\in 
[0,\eps_0]\times\R^3.
\]
We also denote  by $\tilde v(\eps, \vecx) := 
v^\eps(\vecx)$ 
for any $0\le \eps \le \eps_0$ and 
$\vecx\in U^\eps$.
By the definition of material derivative,
we have
\[
\dot{v}
=
\frac{\partial}{\partial\eps}
\tilde v(\eps, 
\mathcal{T}(\eps,\cdot))\Big{|}_{\eps=0},
\quad
\text{in}
\quad
H^1(K).
\]
Applying the chain  rule, we obtain
\begin{align}
\dot v
&
=
\frac{\partial\tilde v}{\partial \eps}
(0,\mathcal{T}(0,\cdot)) 
+
\nabla \tilde v(0,\mathcal{T}(0,\cdot))
\cdot
\frac{\partial \mathcal{T}(0,\cdot)}{\partial 
\eps}
\notag
\\
&
=
\frac{\partial \tilde v(0,\cdot)}{\partial \eps}
+
\nabla v^0
\cdot
V,
\quad
\text{in}
\quad
H^1(K).
\notag
\end{align}
This implies 
\[
v'
=
\frac{\partial \tilde v(0,\cdot)}{\partial \eps}
=
\lim_{\eps\goto 0}
\frac{v^\eps-v^0}{\eps}
\quad
\text{in}
\quad
H^1(K).
\]
\end{proof}

\begin{remark}\label{rem:v sha vx}
The limit in the above lemma does not hold in 
$H^1(U^0)$ 
since,
in general, $v^\eps$ is not defined in $U^0$.
\end{remark}

Similar definitions can be introduced for vector 
functions 
$\vecv$. 
The following lemmas state some useful properties 
of 
material and 
shape derivatives which will be used frequently 
in 
the 
remainder
of the paper.

\begin{lemma}\label{pro:mat sha pro}
Let $ \dot v$, $\dot w$ be material derivatives, 
and
$v'$, $w'$ be shape derivatives of
 $v^\eps$, $w^\eps$ in $H^1(U^\eps)$, $\eps\ge 
0$, 
respectively. 
Then the following statements are true.
\begin{enumerate}[(i)]
\item\label{ite:t1}  
The material and shape derivatives of the product 
$v^\eps 
w^\eps$ 
are $
\dot v w^0
+
v^0\dot w$ and 
$
v' w^0
+
v^0 w'$, respectively.
\item\label{ite:t4}
The material and shape derivatives of the 
quotient 
$v^\eps/ 
w^\eps$ 
are $
(\dot v w^0
-
v^0\dot w)/(w^0)^2$ and 
$
(v' w^0
-
v^0 w')/(w^0)^2$, respectively,
provided that all the fractions are well-defined.
\item\label{ite:t3}
If $v^\eps = v$ for all $\eps\ge 0$, then $ \dot 
v 
= \nabla 
v^0\cdot V
= \nabla v\cdot V$
and $v' = 0$.
\item\label{ite:t2}
If 
\[
\mathcal{J}_1(U^\eps) := 
\displaystyle\int_{U^\eps} 
v^\eps\,d\vecx,
\quad
\mathcal{J}_2(U^\eps) := 
\displaystyle\int_{\G^\eps} 
v^\eps\,d\sigma,
\]
and
\[
d\mathcal{J}_i(U^\eps)|_{\eps=0}
:=
\lim_{\eps\goto 0}
\frac{J_i(U^\eps)-J_i(U^0)}{\eps},
\
i=1,2,
\]
then 
\[
d\mathcal{J}_1(U^\eps)|_{\eps=0}
=
\int_{U^0} v'\,d\vecx
+
\int_{\G^0} v^0
\inpro{V}{{\vecn^0}}
\,d\sigma
\]
and 
\[
d\mathcal{J}_2(U^\eps)|_{\eps=0}
=
\int_{\G^0} v'\,d\sigma
+
\int_{\G^0} 
\left(
\frac{\partial v^0}{\partial n}
+
\divv_{\G^0}(\vecn^0)\,
v^0
\right)
\inpro{V}{{\vecn^0}}
\,d\sigma.
\]
\end{enumerate}
\end{lemma}

\begin{proof}
Statements~{\eqref{ite:t1}--\eqref{ite:t3}} can 
be 
obtained 
by using 
elementary calculations. Statement~\eqref{ite:t2}
is proved in~\cite[pages~113, 116]{SokZol92}.
\end{proof}

The following lemma, which is proved 
in~\cite{ChePhaTra15}, gives the material and 
shape derivatives of the normal field 
$\vecn^\eps$ to the surfaces $\G^\eps$.

\begin{lemma}\label{lem:n shape}
The material and shape derivatives of the normal 
field
$\vecn^\eps$ are given by 
\[
\dot{\vecn} 
=
\vecn' = -\nabla_{\G^0}\kappa.
\]
\end{lemma}

\subsection{Shape derivative for Dirichlet 
conditions}

In this subsection, existence of the shape 
derivative of the solution to heat equation with 
Dirichlet condition will be clarified.
We consider the perturbed initial-boundary value 
problem 
\begin{subequations}\label{e:sh1}
\begin{align}
u_t^\eps
-
\Delta u^\eps
&
=
f^\eps
\quad 
\text{in }
Q_\cT^{\eps}:=
(0,T)\times U^\eps
\label{e:sh1a}
\\
u^\eps 
&
=
0
\quad 
\text{on }
\Sigma_\cT^\eps:=
(0,T)\times \G^\eps 
\label{e:sh1b}
\\
u^\eps|_{t=0}
&
=
g^\eps
\quad 
\text{in }
U^\eps.
\label{e:sh1c}
\end{align}
\end{subequations}
Meanwhile, the reference initial-boundary value 
problem on 
the reference domain $U^0$ is given by 
\begin{subequations}\label{e:sh2}
\begin{align}
u_t^0
-
\Delta u^0
&
=
f^0
\quad 
\text{in }
Q_T^0
\label{e:sh2a}
\\
u^0
&
=
0
\quad 
\text{on }
\Sigma_T^0
\label{e:sh2b}
\\
u^0|_{t=0}
&
=
g^0
\quad 
\text{in }
U^0.
\label{e:sh2c}
\end{align}
\end{subequations}
The weak formulation of~\eqref{e:sh1} reads as 
follows: 
given $f^\eps\in L^2(Q_\cT^{\eps})$ and 
$g^\eps\in 
L^2(U^\eps)$, find 
$u^\eps\in L^2(0,T; H_0^1(U^\eps)) \cap 
C^0([0,T]; 
L^2(U^\eps))$ such that 
\begin{equation}\label{e:weak for}
\begin{cases}
\frac{d}{dt}\inpro{u^\eps(t)}{v}_{L^2(U^\eps)}
+
a(u^\eps(t),v; U^\eps)
=
\inprod{f^\eps(t)}{v}_{L^2(U^\eps)}
\quad 
\forall v\in H_0^1(U^\eps)
\\
u^\eps(0)
=
g^\eps,
\end{cases}
\end{equation}
where the bilinear form $a(\cdot,\cdot; U)$ 
associated with a domain $U$ is defined 
by 
\begin{equation}\label{e:a bilinear}
a(v,w; U)
:=
\int_U
\nabla v\cdot\nabla w\,d\vecx,
\quad 
v,w\in H^1(U).
\end{equation}
In this paper, we assume that the sequences 
$\sett{f^\eps}_{0<\eps<\eps_0}$ and 
$\sett{g^\eps}_{0<\eps<\eps_0}$ satisfy
\begin{subequations}\label{fg eps cond}
\begin{align}
\lim_{\eps\goto 0}
\norm{f^\eps\dcirc \cT^\eps - f^0}{L^2(0,T; 
L^2(U^0))}
&
=
0,
\label{feps f}
\\
\lim_{\eps\goto 0}
\norm{g^\eps\circ\cT^\eps - g^0}{H^1(U^0)}
&
=
0.
\label{geps g}
\end{align}
\end{subequations}
The above equations suggest that the 
$\sett{\norm{f^\eps\dcirc \cT^\eps}{L^2(0,T; 
L^2(U^0))}}_{\eps>0}$ and 
$\sett{\norm{g^\eps\circ\cT^\eps}{H^1(U^0)}}_{
\eps>0}$ are bounded for sufficiently small 
$\eps$.
This together with the assumption that $V\in 
W^{1,\infty}(U_+)$ implies that for sufficiently 
small $\eps$, there hold
\begin{subequations}\label{fg eps1}
\begin{align}
\norm{f^\eps}{L^2(0,T; 
L^2(U^\eps))}
\le 
M_1,
\qquad
&
\norm{f^\eps\dcirc \cT^\eps}{L^2(0,T; 
L^2(U^0))}
\le 
M_1,
\label{feps 1}
\\
\norm{g^\eps}{H^1(U^\eps)}
\le 
M_2,
\qquad
&
\norm{g^\eps\circ\cT^\eps}{H^1
(U^0)}
\le
M_2,
\label{geps 1}
\end{align}
\end{subequations}
where $M_1$ and $M_2$ are two positive 
constants depending on $f^0$ and $g^0$. 

The following lemmas (see~\cite[Page 
366]{QuaVal97}) state  
the unique existence of the weak solution to the 
initial-boundary value problem with the Dirichlet 
boundary 
condition.

\begin{lemma}\label{l:unique Dir}
Given $f^\eps\in L^2(Q_\cT^{\eps})$ and 
$g^\eps\in 
L^2(U^\eps)$, there 
exists a unique solution $u^\eps\in L^2(0,T; 
H_0^1(U^\eps))\cap C^0([0,T]; 
L^2(U^\eps))$ to~\eqref{e:weak for}. Moreover, 
$\partial 
u^\eps/\partial t\in L^2(0,T; H^{-1}(U^\eps))$ 
and 
the 
energy estimate 
\begin{equation}\label{e:ene est}
\begin{aligned}
\norm{u^\eps(t)}{L^2(U^\eps)}^2
+
\int_0^t
\norm{u^\eps(\tau)}{H^1(U^\eps)}^2\,d\tau
&
\lesssim 
\norm{g^\eps}{L^2(U^\eps)}^2
+
\int_0^t
\norm{f^\eps(\tau)}{L^2(U^\eps)}^2\,d\tau
\end{aligned}
\end{equation}
holds for each $t\in [0,T]$.
Here, the constant implicitly included in the 
above 
inequality is independent of $T$.
\end{lemma}

\begin{lemma}\label{l:unique Dir 2}
Given $f^\eps\in L^2(Q_\cT^{\eps})$ and 
$g^\eps\in 
H_0^1(U^\eps)$ for 
all sufficiently small $\eps>0$,  the 
solution 
$u^\eps$ 
to~\eqref{e:weak for} belongs to $L^\infty(0,T; 
H_0^1(U^\eps))\cap H^1(0,T; L^2(U^\eps))$ and 
satisfies 
\begin{equation}\label{e:ene est 2}
\begin{aligned}
\sup_{t\in (0,T)}
\norm{u^\eps(t)}{H^1(U^\eps)}^2 
+
\int_0^T 
\norm{u_t^\eps(t)}{L^2(U^\eps)}^2\,dt
&
\lesssim 
\norm{g^\eps}{H^1(U^\eps)}^2 
+
\int_0^T 
\norm{f^\eps(t)}{L^2(U^\eps)}^2\,dt,
\end{aligned}
\end{equation}
where the constant implicitly included in the 
above 
inequality is independent of $T$.
\end{lemma}

The results in Lemmas~\ref{l:unique Dir} 
and~\ref{l:unique Dir 2},
and~\eqref{fg eps1} yield the following lemma.
\begin{lemma}\label{l:9Sep1d}
Assume that the conditions~\eqref{fg eps cond} 
and the assumptions in Lemma~\ref{l:unique Dir 2} 
are satisfied. There hold 
\begin{subequations}\label{9Sep1c}
\begin{align}
\sup_{t\in(0,T)}
\norm{u^\eps(t)}{L^2(U^\eps)}^2
+
\norm{u^\eps}{L^2(0,T;H^1(U^\eps))}^2
\lesssim 
M_3^2,
\label{9Sep1a}
\\
\sup_{t\in(0,T)}
\norm{u^\eps(t)}{H^1(U^\eps)}^2
+
\norm{u_t^\eps}{L^2(0,T; L^2(U^\eps))}^2 
\lesssim 
M_4^2,
\label{9Sep1b}
\end{align}
\end{subequations}
where $M_3$ and $M_4$ are positive 
constands depending on $f^0$ and $g^0$.
\end{lemma}

In this paper, we shall frequently use the 
following results.

\begin{lemma}\label{l:9Oct1a}
Let $v^\eps, w^\eps \in L^2(0,T;L^2(U))$ satisfy
\[
\lim_{\eps\goto 0}
\norm{v^\eps}{L^2(0,T;L^2(U))}
=
0
\quad 
\text{and}
\quad 
\norm{w^\eps}{L^2(0,T;L^2(U))}
\le 
M
\
\text{for sufficiently small } \eps.
\]
There holds 
\[
\lim_{\eps\goto 0}
\int_0^t
\int_U 
v^\eps(\tau,\vecx)\, 
w^\eps(\tau,\vecx)\,d\vecx\,d\tau
=
0
\quad 
\forall 
t\in [0,T].
\]
\end{lemma}

\begin{proof}
We have 
\begin{align}
\abs{ 
\int_0^t
\int_U 
v^\eps(\tau,\vecx)\, 
w^\eps(\tau,\vecx)\,d\vecx\,d\tau
}
&
\le
\norm{v^\eps}{L^2(0,T;L^2(U))}
\norm{w^\eps}{L^2(0,T;L^2(U))}
\notag
\\
&
\le 
M\norm{v^\eps}{L^2(0,T;L^2(U))}.
\notag
\end{align}
Letting $\eps$ go to zero, we obtain the desired 
equality.
\end{proof}

\begin{lemma}\label{l:fv conv}
Let $v^\eps\in L^\infty(U)$ and $w^\eps\in 
L^2(0,T;L^2(U))$ satisfy 
\[
\lim_{\eps\goto 0}
\norm{v^\eps}{L^\infty(U)} = 0
\quad 
\text{and}
\quad 
\norm{w^\eps}{L^2(0,T;L^2(U))}
\le M
\quad 
\text{for sufficiently small } \eps.
\]
There holds 
\[
\lim_{\eps\goto 0}
\int_0^t 
\int_U
v^\eps(\vecx) w^\eps(\tau,\vecx)\,d\vecx\,d\tau
=
0
\quad 
\forall t\in [0,T].
\]
\end{lemma}

\begin{proof}
The proof can be done in the same manner as the 
proof in Lemma~\ref{l:9Oct1a}.
\end{proof}

\begin{lemma}\label{l:Oct2e}
Let $v^\eps, v^0\in L^\infty(U)$ and $w^\eps, 
w^0\in L^2(0,T;L^2(U))$ satisfy
\begin{align}
\lim_{\eps\goto 0}
\norm{v^\eps - v^0}{L^\infty(U)} 
=
0
\quad 
\text{and}
\quad 
\lim_{\eps\goto 0}
\norm{w^\eps - w^0}{L^2(0,T;L^2(U))} 
=
0.
\label{9Oct2e}
\end{align}
There holds 
\begin{align}
\lim_{\eps\goto 0}
\norm{v^\eps w^\eps - v^0 w^0}{L^2(0,T;L^2(U))} 
=
0.
\label{9Oct2f}
\end{align}
\end{lemma}

\begin{proof}
Applying the triangle inequality, we have 
\begin{align}
\norm{v^\eps w^\eps - v^0 w^0}{L^2(0,T;L^2(U))}^2
&
\lesssim 
\norm{\brac{v^\eps - v^0} 
w^\eps}{L^2(0,T;L^2(U))}^2 
+
\norm{v^0\brac{w^\eps - w^0} 
}{L^2(0,T;L^2(U))}^2 
\notag
\\
&
\le 
\norm{v^\eps-v^0}{L^\infty(U)}^2 
\norm{w^\eps}{L^2(0,T;L^2(U))}^2 
+
\norm{v^0}{L^\infty}^2 
\norm{w^\eps-w^0}{L^2(0,T;L^2(U))}^2.
\notag
\end{align}
Letting $\eps$ go to zero and 
noting~\eqref{9Oct2e}, we obtain~\eqref{9Oct2f}.
\end{proof}

The following lemma states the convergence of 
$u^\eps$ and its derivative 
$u_t^\eps$ to $u^0$ and $u_t^0$, respectively.

\begin{lemma}\label{l:u u0 1}
Assume that the assumptions in 
Lemma~\ref{l:unique 
Dir 2} 
are satisfied.
Let $u^\eps$ and $u^0$ be solutions 
to~\eqref{e:sh1} 
and~\eqref{e:sh2}, respectively.
There holds 
\begin{subequations}\label{e:u u0 1}
\begin{align}
\lim_{\eps\goto 0}
\sup_{t\in [0,T]}
\norm{(u^\eps\dcirc \cT^\eps-u^0)(t)}{L^2(U^0)}^2
&
=
0,
\label{lim 1}
\\
\lim_{\eps\goto 0}
\norm{u^\eps\dcirc 
\cT^\eps-u^0}{L^2(0,T;H^1(U^0))}
&
=
0,
\label{lim 2}
\\
\lim_{\eps\goto 0}
\norm{ 
u_t^\eps\dcirc \cT^\eps- u_t^0
}{L^2(0,T;H^{-1}(U^0))}^2d\tau
&
=
0.
\label{lim 3}
\end{align}
\end{subequations}

\end{lemma}
 
\begin{proof}
The first equation in~\eqref{e:weak for} gives 
\begin{align}
\int_{U^\eps}
u_t^\eps(t,\vecx)\, v^\eps(\vecx)\,d{\vecx} 
+
\int_{U^\eps}
\nabla u^\eps(t,\vecx) \cdot \nabla 
v^\eps(\vecx)\,d{\vecx}
=
\int_{U^\eps}
f^\eps(t,\vecx)\, v^\eps(\vecx)\,d{\vecx}
\quad 
v^\eps\in 
C_0^\infty(U^\eps) 
\label{D eps 1}
\end{align}
for $0\le t\le T$.
In particular,
\begin{align}
\int_{U^0}
u_t^0(t,\vecx)\, w(\vecx)\,d{\vecx} 
+
\int_{U^0}
\nabla u^0(t,\vecx) \cdot \nabla 
w(\vecx)\,d{\vecx}
=
\int_{U^0}
f^0(t,\vecx)\, w(\vecx)\,d{\vecx}
\quad 
\forall w\in C_0^\infty(U^0). 
\label{D 2}
\end{align}
Using the change of variables $\vecx = 
\cT^\eps(\vecy)$
in~\eqref{D eps 1}, we have 
\begin{align}
\int_{U^0} 
u_t^\eps\dot{\circ}\cT^\eps(t,\vecy)\,
v^\eps{\circ} \cT^\eps(\vecy)\,
\gamma(\eps,\vecy)
\,d\vecy
&
+
\int_{U^0}
[\nabla (v^\eps{\circ} \cT^\eps)(\vecy)]^T\,
A(\eps,\vecy)\,
\nabla(u^\eps\dot{\circ}\cT^\eps)(t,\vecy)\,d\vecy
\notag
\\
&
=
\int_{U^0}
f^\eps\dot{\circ}\cT^\eps(t,\vecy)\,
v^\eps{\circ} \cT^\eps(\vecy)\,
\gamma(\eps,\vecy)\,
d\vecy.
\notag
\end{align}
Since $C_0^\infty(U^\eps)$ is dense in 
$H_0^1(U^\eps)$,
this is true for all $v^\eps\in H_0^1(U^\eps)$. 
Thus, 
for all $w\in 
H_0^1(U^0)$ there holds
\begin{align}
\int_{U^0} 
u_t^\eps\dot{\circ}\cT^\eps(t,\vecy)\,
w(\vecy)\,
\gamma(\eps,\vecy)
\,d\vecy
&
+
\int_{U^0}
[\nabla w(\vecy)]^T\,
A(\eps,\vecy)\,
\nabla(u^\eps\dot{\circ}\cT^\eps)(t,\vecy)\,d\vecy
\notag
\\
&
=
\int_{U^0}
f^\eps\dot{\circ}\cT^\eps(t,\vecy)\,
w(\vecy)\,
\gamma(\eps,\vecy)\,
d\vecy.
\label{D 1}
\end{align}
Subtracting~\eqref{D 2} from~\eqref{D 1} we 
deduce 
\begin{align}
\int_{U^0}
\big(
u_t^\eps\dot{\circ}\cT^\eps
-
u_t^0
\big)(t,\vecy)\,
 w(\vecy)\,d\vecy
\
+
&
\int_{U^0}
[\nabla w(\vecy)]^T \,
\nabla\brac{u^\eps\dot{\circ}\cT^\eps - 
u^0}(t,\vecy)
\,d\vecy
\notag
\\
=
&
- 
\int_{U^0}
u_t^\eps\dot{\circ}\cT^\eps(t,\vecy) \,
w(\vecy) 
\brac{ 
\gamma(\eps,\vecy) - 1
}\,
d\vecy
\notag
\\
&
-
\int_{U^0}
[\nabla w(\vecy)]^T \,
\left(
A(\eps,\vecy) - I
\right)
\,
\nabla\brac{u^\eps\dot{\circ}\cT^\eps}(t,\vecy)\,
d\vecy
\notag 
\\
&
+
\int_{U^0}
\brap{
f^\eps\dot{\circ}\cT^\eps(t,\vecy)
\, \gamma(\eps,\vecy) 
-
f^0(t,\vecy)
}\,
w(\vecy)\,d\vecy.
\label{e:u eps u D0}
\end{align}
For any $t\in [0,T]$, we choose in~\eqref{e:u eps 
u D0} 
$w(\vecy) = 
\brac{u^\eps\dot{\circ}\cT^\eps- u^0}(t,\vecy)$ 
to 
obtain

\begin{align}
\frac12 
\frac{d}{dt}
\int_{U^0} 
\left[ 
\big(u^\eps\dot{\circ}\cT^\eps
\right.
&
\left. 
-
u^0\big)(t,\vecy)
\right]^2 
d\vecy
+
\int_{U^0}
\norm{\nabla\brac{u^\eps\dot{\circ}\cT^\eps - 
u^0}(t,\vecy)}{}^2\,d\vecy
\notag
\\
=
&
-\int_{U^0} 
u_t^\eps\dot{\circ}\cT^\eps(t,\vecy)\,
\brac{u^\eps\dot{\circ}\cT^\eps-u^0}(t,\vecy)\,
\brac{\gamma(\eps,\vecy)-1}\,d\vecy
\notag
\\
&
-\int_{U^0} 
\Big[
\nabla 
\brac{
u^\eps\dot{\circ}\cT^\eps - u^0
}(t,\vecy)
\Big]^T
\big( A(\eps,\vecy)-I\big) 
\,\nabla\brac{u^\eps\dot{\circ}\cT^\eps}(t,\vecy)
\,d\vecy
\notag
\\
&
+
\int_{U^0} 
\left[ 
f^\eps\dot{\circ}\cT^\eps(t,\vecy) 
\gamma(\eps,\vecy) 
- 
f^0(t,\vecy)
\right]
\brac{
u^\eps\dot{\circ}\cT^\eps - u^0
}
(t,\vecy)\,d\vecy.
\label{e:A1}
\end{align}
Integrating both sides of~\eqref{e:A1} over 
$[0,t]$ for any 
$t\in [0,T]$ 
and noting the initial conditions 
$u^\eps\dot{\circ}\cT^\eps(0,\vecy) = 
g^\eps{\circ}\cT^\eps(\vecy)$ and 
$u^0(0,\vecy) = g^0(\vecy)$ for 
all $\vecy\in U^0$,
we 
obtain
\begin{align}
\frac12 
\int_{U^0} 
[ 
(u^\eps\dot{\circ}\cT^\eps
&
-
u^0)(t,\vecy)
]^2 
\,d\vecy
+
\int_0^t
\int_{U^0}
\norm{\nabla\brac{u^\eps\dot{\circ}\cT^\eps - 
u^0}(\tau,\vecy)}{}^2\,d\vecy\,d\tau
\notag
\\
=
&\
\frac12 
\int_{U^0}
[ 
g^\eps{\circ}\cT^\eps(\vecy)
-
g^0(\vecy)
]^2 
\,d\vecy
\notag
\\
&
-
\int_0^t
\int_{U^0} 
u_t^\eps\dot{\circ}\cT^\eps(\tau,\vecy)\,
 \brac{u^\eps\dot{\circ}\cT^\eps - 
u^0}{(}\tau,\vecy) \,
\brac{\gamma(\eps,\vecy)-1}\,d\vecy\,d\tau
\notag
\\
&
-
\int_0^t
\int_{U^0} 
\Big[
\nabla 
\brac{
u^\eps\dot{\circ}\cT^\eps - u^0
}(\tau,\vecy)
\Big]^T
\big( A(\eps,\vecy)-I\big) 
\nabla\brac{u^\eps\dot{\circ}\cT^\eps}(\tau,
\vecy)\,d\vecy\,
d\tau
\notag
\\
&
+
\int_0^t
\int_{U^0} 
\left[ 
f^\eps\dot{\circ}\cT^\eps(\tau,\vecy)
\gamma(\eps,\vecy) 
- 
f^0(\tau,\vecy)
\right]
\brac{
u^\eps\dot{\circ}\cT^\eps - u^0
}
(\tau,\vecy)\,d\vecy\,d\tau.
\label{e:A3}
\end{align}
Applying  Lemma~\ref{l:Poincare inequality} and 
the triangle inquality, 
we derive 
\begin{align}
\left\| (u^\eps\dot{\circ}\cT^\eps 
\right. 
&
\left.
-\ u^0)
(t)
\right\|_{L^2(U^0)}^2
+
\int_0^t 
\norm{(u^\eps\dot{\circ}\cT^\eps - 
u^0)(\tau)}{H^1(U^0)}^2\,d\tau
\lesssim
\norm{g^\eps{\circ} \cT^\eps - g^0}{L^2(U^0)}^2
\notag
\\
&
+
\abs{
\int_0^t
\int_{U^0} 
u_t^\eps\dot{\circ}\cT^\eps(\tau,\vecy)\,
 \brac{u^\eps\dot{\circ}\cT^\eps - 
u^0}{(}\tau,\vecy) \,
\brac{\gamma(\eps,\vecy)-1}\,d\vecy\,d\tau
}
\notag
\\
&
+
\abs{
\int_0^t
\int_{U^0} 
\Big[
\nabla 
\brac{
u^\eps\dot{\circ}\cT^\eps - u^0
}(\tau,\vecy)
\Big]^T
\big( A(\eps,\vecy)-I\big) 
\nabla\brac{u^\eps\dot{\circ}\cT^\eps}(\tau,
\vecy)\,d\vecy\,
d\tau
}
\notag
\\
&
+
\abs{
\int_0^t
\int_{U^0} 
\left[ 
f^\eps\dot{\circ}\cT^\eps(\tau,\vecy)
\gamma(\eps,\vecy) 
- 
f^0(\tau,\vecy)
\right]
\brac{
u^\eps\dot{\circ}\cT^\eps - u^0
}
(\tau,\vecy)\,d\vecy\,d\tau
}.
\label{e:A18}
\end{align}
Noting~\eqref{feps f}, the result in 
Lemma~\ref{l:9M1a} and applying triangle 
inequality we deduce 
\begin{align}
\lim_{\eps\goto 0}
\norm{f^\eps\dcirc \cT^\eps \gamma(\eps,\cdot) - 
f^0}{L^2(0,T;L^2(U^0))}
=
0.
\label{9Oct1b}
\end{align}
Letting $\eps$ go to zero, noting~\eqref{geps g}, 
\eqref{9M1h}, \eqref{e:st 1}, \eqref{9Oct1b}, 
\eqref{9Sep1c} and applying Lemmas~\ref{l:9Oct1a} 
and~\ref{l:fv conv}, we obtain
\begin{align}
\lim_{\eps\goto 0}
\sup_{t\in [0,T]}
\norm{(u^\eps\dcirc \cT^\eps-u^0)(t)}{L^2(U^0)}
&
=
0
\notag
\\
\lim_{\eps\goto 0}
\norm{u^\eps\dcirc 
\cT^\eps-u^0}{L^2(0,T;H^1(U^0))}
&
=
0.
\label{esi8}
\end{align}
We shall next prove that 
$\displaystyle\lim_{\eps\goto 0} 
\norm{u_t^\eps{\circ}\cT^\eps - u_t^0}{L^2(0,T; 
H^{-1}(U^0))}=0$.
From~\eqref{e:u eps u D0}, for any $\tau\in 
[0,T]$ we have 
\begin{align}
&
\int_{U^0}
\left[
u_t^\eps\dot{\circ}\cT^\eps(\tau,\vecy) 
\right.
-
\left.
u_t^0(\tau,\vecy)
\right]
w(\vecy)\,d\vecy
\notag
\\
\le
&
\abs{\int_{U^0}
[\nabla w(\vecy)]^T \,
\nabla\brac{u^\eps\dot{\circ}\cT^\eps - 
u^0}(t,\vecy)
\,d\vecy}
+
\abs{ 
\int_{U^0}
u_t^\eps\dot{\circ}\cT^\eps(t,\vecy) \,
w(\vecy) 
\brac{ 
\gamma(\eps,\vecy) - 1
}\,
d\vecy
}
\notag
\\
+
&
\abs{ 
\int_{U^0}
[\nabla w(\vecy)]^T \,
\left(
A(\eps,\vecy) - I
\right)
\,
\nabla\brac{u^\eps\dot{\circ}\cT^\eps}(t,\vecy)\,
d\vecy
}
+
\abs{
\int_{U^0}
\brap{
f^\eps\dot{\circ}\cT^\eps(t,\vecy)
\, \gamma(\eps,\vecy) 
-
f^0(t,\vecy)
}\,
w(\vecy)\,d\vecy
}.
\notag
\end{align}
This gives
\begin{align}
&
\int_{U^0}
\left[
u_t^\eps\dot{\circ}\cT^\eps(\tau,\vecy) 
\right.
-
\left.
u_t^0(\tau,\vecy)
\right]
w(\vecy)\,d\vecy
\notag
\\
\le\
&
\norm{w}{H^1(U^0)}
\Big{(}
\norm{(u^\eps\dot{\circ}\cT^\eps - 
u^0)(\tau)}{H^1(U^0)}
+
\norm{\gamma(\eps,\cdot) - 1}{L^\infty(U^0)}
\norm{u_t^\eps{\circ}\cT^\eps(\tau)}{L^2(U^0)}
\notag
\\
&
+
\norm{A(\eps,\cdot) - I}{L^\infty(U^0)}
\norm{u^\eps\dot{\circ}\cT^\eps(\tau)}{H^1(U^0)}
+
\norm{f^\eps{\dcirc}\cT^\eps(\tau) 
\gamma(\eps,\cdot) 
-
f^0(\tau)
}{L^2(U^0)}
\Big{)}.
\label{e:A16}
\end{align}
This is true for all $w\in H_0^1(U^0)$. Thus, for 
any $\tau\in 
[0,T]$, there holds
\begin{align}
\norm{ 
(u_t^\eps\dcirc \cT^\eps- u_t^0)(\tau)
}{H^{-1}(U^0)}
&
\lesssim 
\norm{(u^\eps\dot{\circ}\cT^\eps - 
u^0)(\tau)}{H^1(U^0)}
\notag
\\
&
+
\norm{\gamma(\eps,\cdot) - 1}{L^\infty(U^0)}
\norm{u_t^\eps\dcirc\cT^\eps(\tau)}{L^2(U^0)}
\notag 
\\
&
+
\norm{A(\eps,\cdot) - I}{L^\infty(U^0)}
\norm{u^\eps\dcirc \cT^\eps(\tau)}{H^1(U^0)}
\notag 
\\
&
+
\norm{f^\eps{\dcirc}\cT^\eps(\tau) 
\gamma(\eps,\cdot) 
-
f^0(\tau)
}{L^2(U^0)}.
\label{esi5}
\end{align}
Squaring up both sides, using the Cauchy--Schwarz
inequality and then integrating over 
$[0,T]$, we 
obtain 
\begin{align}
\norm{ 
u_t^\eps\dcirc \cT^\eps- u_t^0
}{L^2(0,T;H^{-1}(U^0))}^2
&
\lesssim 
\norm{u^\eps\dot{\circ}\cT^\eps - 
u^0}{L^2(0,T;H^1(U^0))}^2
\notag 
\\
&
+
\norm{\gamma(\eps,\cdot) - 1}{L^\infty(U^0)}^2
\norm{u_t^\eps\dcirc 
\cT^\eps}{L^2(0,T;L^2(U^0))}^2
\notag 
\\
&
+
\norm{A(\eps,\cdot) - I}{L^\infty(U^0)}^2
\norm{u^\eps\dcirc\cT^\eps}{L^2(0,T;H^1(U^0))}^2
\notag
\\
&
+
\norm{f^\eps{\dcirc}\cT^\eps
\gamma(\eps,\cdot) 
-
f^0
}{L^2(0,T;L^2(U^0))}^2.
\notag
\end{align}
It follows from~\eqref{esi8}, \eqref{9M1h}, 
\eqref{e:st 1}, \eqref{9Sep1c} and~\eqref{9Oct1b} 
that 
\[
\lim_{\eps\goto 0}
\norm{ 
(u_t^\eps\dcirc \cT^\eps- u_t^0)(\tau)
}{L^2(0,T;H^{-1}(U^0))}
=
0,
\]
finishing the proof of the theorem.
\end{proof}

\begin{lemma}
Let $f^\eps\in L^2(Q_T^\eps)$ satisfy 
\begin{align}
\lim_{\eps\goto 0}
\norm{\frac{f^\eps\dcirc \cT^\eps - 
f^0}{\eps}-\nabla f^0\cdot V}{L^2(0,T;L^2(U^0))}
=
0.
\label{9Oct2a}
\end{align}
Then, there holds 
\begin{align}
\lim_{\eps\goto 0}
\norm{
\frac{
f^\eps{\dcirc} \cT^\eps
\, \gamma(\eps,\cdot) 
-
f^0
}{\eps}
-
\divv\left(V f^0\right)
}{L^2(0,T;L^2(U^0))}
&
=
0.
\label{9Oct2b}
\end{align}
\end{lemma}

\begin{proof}
We first note that if~\eqref{9Oct2a} is 
satisfied, there holds
\begin{align}
\lim_{\eps\goto 0}
\norm{f^\eps\dcirc \cT^\eps - 
f^0}{L^2(0,T;L^2(U^0))}
=
0.
\label{9Oct2d}
\end{align}
The triangle inequality gives
\begin{align}
\Big{\|}
\frac{
f^\eps{\dcirc} \cT^\eps
\, \gamma(\eps,\cdot)
-
f^0
}{\eps}
-
\divv\left(V f^0\right)
\Big{\|}_{L^2(0,T; L^2(U^0))}
&
\lesssim
\norm{ 
\frac{\gamma(\eps,\cdot)-1}{\eps}
(f^\eps{\dcirc}\cT^\eps) 
-
f^0\divv V
}{L^2(0,T; L^2(U^0))}
\notag
\\
&
+
\norm{ 
\frac{ 
f^\eps{\dcirc} \cT^\eps - f^0
}{\eps}
-
V\cdot \nabla f^0
}{L^2(0,T; L^2(U^0))}.
\notag
\end{align}
Applying Lemma~\ref{l:fv conv}, and 
noting~\eqref{9Oct2d} and~\eqref{9Oct2c},
we prove that
the first norm on the right hand side of the 
above inequality converges to zero when $\eps$ goes to zero. This together 
with the assumption~\eqref{9Oct2a} 
yields~\eqref{9Oct2b}.

\end{proof}

\begin{lemma}\label{l:9Au1a}
Assume that $f^\eps\in L^2(Q_T^\eps)$ and 
$g^\eps\in 
H_0^2(U^\eps)$ for all sufficiently small 
$\eps\ge 0$.
Let $u^\eps$ and $u^0$ be solutions 
to~\eqref{e:sh1} 
and~\eqref{e:sh2}, respectively.
There hold
\begin{align}
\lim_{\eps\goto 0}
\norm{
u_t^\eps{\dcirc}\cT^\eps
\frac{
\gamma(\eps,\cdot) - 1
}{\eps}
-
u_t^0\divv V
}{L^2(0,T;H^{-1}(U^0))}
&
=
0
\label{Apr19e}
\\
\lim_{\eps\goto 0}
\norm{ 
\frac{
A(\eps,\cdot) - I
}{\eps}
\,
\nabla(u^\eps{{\dcirc}} \cT^\eps)
- 
A'(0,\cdot) \nabla u^0}{L^2(0,T;L^2(U^0))}
&
=
0
\label{Apr19f}
\end{align}

\end{lemma}

\begin{proof}
The triangle inequality gives
\begin{align}
\norm{
u_t^\eps{\dcirc}\cT^\eps
\frac{
\gamma(\eps,\cdot) {-} 1
}{\eps}
{-}
u_t^0\divv V
}{L^2(0,T;H^{-1}(U^0))}^2
&
\lesssim
\norm{
\left(
u_t^\eps{\dcirc}\cT^\eps
-
u_t^0
\right)   
\frac{
\gamma(\eps,\cdot) - 1
}{\eps}
}{L^2(0,T;H^{-1}(U^0))}^2
\notag
\\
&
+
\norm{
u_t^0
\left(
\frac{
\gamma(\eps,\cdot) - 1
}{\eps}
-
\divv V
\right)   
}{L^2(0,T;H^{-1}(U^0))}^2.
\label{9M1n}
\end{align}
Using the duality argument and 
noting~\eqref{e:vbn 4},
for every $\tau\in [0,T]$
we 
have 
\begin{align}
\Big\|
(
u_t^\eps{\dcirc}\cT^\eps
&
-
u_t^0)(\tau)
\frac{
\gamma(\eps,\cdot) {-} 1
}{\eps}
\Big\|_{H^{-1}(U^0)}
\notag
\\
&
=
\sup_{w\in H_0^1(U^0)\atop w\not= 0}
\frac{ 
\displaystyle
\int_{U^0} 
\left(
u_t^\eps{\dcirc}\cT^\eps
-
u_t^0
\right)   
(\tau,\vecy)
\frac{
\gamma(\eps,\vecy) {-} 1
}{\eps}
w(\vecy)\,d\vecy
}{
\norm{w}{H^1(U^0)}
}
\notag 
\\
&
\le 
\norm{\frac{
\gamma(\eps,\vecy) {-} 1
}{\eps}
}{L^\infty(U^0)}
\sup_{w\in H_0^1(U^0)\atop w\not= 0}
\frac{ 
\displaystyle
\int_{U^0} 
\left(
u_t^\eps{\dcirc}\cT^\eps
-
u_t^0
\right)(\tau,\vecy) 
w(\vecy)\,d\vecy
}{
\norm{w}{H^1(U^0)}
}
\notag 
\\
&
\le 
\norm{\gamma_1 + 
\eps\gamma_2+\eps^2\gamma_3}{L^\infty(U^0)}
\norm{
(u_t^\eps{\circ}\cT^\eps
-
u_t^0)(\tau)
}{H^{-1}(U^0)}. 
\notag
\end{align}
This implies 
\begin{align}
\norm{
\left(
u_t^\eps{\dcirc}\cT^\eps
-
u_t^0
\right)   
\frac{
\gamma(\eps,\cdot) - 1
}{\eps}
}{L^2(0,T;H^{-1}(U^0))}
&
\le 
\norm{\gamma_1 + 
\eps\gamma_2+\eps^2\gamma_3}{L^\infty(U^0)}
\norm{
u_t^\eps{\circ}\cT^\eps
-
u_t^0
}{L^2(0,T;H^{-1}(U^0))}. 
\label{Apr19d}
\end{align}
Similiar arguments give
\begin{align}
\Big\|
u_t^0
\Big( 
\frac{\gamma(\eps,\cdot)-1}{\eps}
-
\divv V
\Big)
\Big\|_{L^2(0,T;H^{-1}(U^0))
}
&
\le
\norm{
\frac{\gamma(\eps,\cdot)-1}{\eps}
-
\divv V
}{L^\infty(U^0)}
\norm{u_t^0}{L^2(0,T;H^{-1}(U^0))},
\label{9M1o}
\end{align}
It follows from~\eqref{9M1n}--~\eqref{9M1o}, \eqref{lim 3} and~\eqref{9Oct2c}
that 
\begin{align}
\lim_{\eps\goto 0}
\norm{
u_t^\eps{\dcirc}\cT^\eps
\frac{
\gamma(\eps,\cdot) {-} 1
}{\eps}
-
u_t^0\divv V
}{L^2(0,T;H^{-1}(U^0))}
&
=
0,
\notag
\end{align}
proving~\eqref{Apr19e}. 

Applying the triangle inequality again, we have 
\begin{align}
\Big{\|} 
\frac{
A(\eps,\cdot) - I
}{\eps}
\,
\nabla(u^\eps{{\dcirc}} \cT^\eps)
&
- 
A'(0,\cdot) \nabla 
u^0\Big{\|}_{L^2(0,T;L^2(U^0))}^2
\lesssim
\Big{\|} 
\frac{
A(\eps,\cdot) - I
}{\eps}
\,
\nabla(u^\eps{{\dcirc}} \cT^\eps - u^0)
\Big{\|}_{L^2(0,T;L^2(U^0))}^2
\notag
\\
&
+
\Big{\|} 
\brac{
\frac{
A(\eps,\cdot) - I
}{\eps}
- 
A'(0,\cdot)
}
\nabla 
u^0\Big{\|}_{L^2(0,T;L^2(U^0))}^2.
\label{Apr19h}
\end{align}
For every $\tau\in [0,T]$, we have 
\begin{align}
\Big{\|} 
\frac{
A(\eps,\cdot) - I
}{\eps}
\,
\nabla(u^\eps{{\dcirc}} \cT^\eps - 
u^0)(\tau)
\Big{\|}_{L^2(U^0)}
&
\lesssim 
\norm{ 
\frac{
A(\eps,\cdot) - I
}{\eps}
}{L^\infty(U^0)}
\norm{ 
\nabla(
u^\eps\dot{\circ}\cT^\eps - u^0
)(\tau)
}{L^2(U^0)}
\notag
\\
&
\le
\norm{ 
\frac{
A(\eps,\cdot) - I
}{\eps}
}{L^\infty(U^0)}
\norm{ 
(
u^\eps\dot{\circ}\cT^\eps - u^0
)(\tau)
}{H^1(U^0)},
\label{Apr19i}
\end{align}
and 
\begin{align}
\Big{\|} 
\brac{
\frac{
A(\eps,\cdot) - I
}{\eps}
- 
A'(0,\cdot)
}
\nabla 
u^0(\tau)\Big{\|}_{L^2(U^0)}
&
\lesssim
\norm{ 
\frac{
A(\eps,\cdot) - I
}{\eps}
-
A'(0,\cdot)
}{L^\infty(U^0)}
\norm{\nabla 
u^0(\tau)}{L^2(U^0)}
\notag
\\
&
\le 
\norm{ 
\frac{
A(\eps,\cdot) - I
}{\eps}
-
A'(0,\cdot)
}{L^\infty(U^0)}
\norm{
u^0(\tau)}{H^1(U^0)}.
\label{Apr19j}
\end{align}
Note that the inequalities~\eqref{Apr19i} and 
\eqref{Apr19j} are true for every $\tau\in [0,T]$.
This together with~\eqref{Apr19h} 
implies 
\begin{align}
\Big{\|} 
\frac{
A(\eps,\cdot) - I
}{\eps}
\,
\nabla(u^\eps{{\dcirc}} \cT^\eps)
&
- 
A'(0,\cdot) \nabla 
u^0\Big{\|}_{L^2(0,T;L^2(U^0))}^2
\lesssim
\norm{ 
\frac{
A(\eps,\cdot) - I
}{\eps}
}{L^\infty(U^0)}^2
\norm{
u^\eps\dot{\circ}\cT^\eps - u^0
}{L^2(0,T;H^1(U^0))}^2
\notag
\\
&
+
\norm{ 
\frac{
A(\eps,\cdot) - I
}{\eps}
-
A'(0,\cdot)
}{L^\infty(U^0)}^2
\norm{
u^0}{L^2(0,T;H^1(U^0))}^2.
\notag
\end{align}
Letting $\eps$ go to zero on both sides and 
noting~\eqref{e:st 2} and~\eqref{lim 2}, we 
obtain~\eqref{Apr19f}, finishing the proof of the 
lemma.
\end{proof}

\begin{lemma}\label{l:material der}
Assume that $f^\eps\in L^2(Q_\cT^{\eps})$ and 
$g^\eps\in 
H_0^2(U^\eps)$ for all sufficiently small 
$\eps>0$.
Assume further that 
\begin{align}
\lim_{\eps\goto 0}
\norm{\frac{f^\eps\dcirc \cT^\eps - f^0}{\eps} 
- \nabla f^0\cdot V}{L^2(0,T:L^2(U^0))} = 0
\label{9Nov1a}
\end{align}
and 
\begin{align}
\lim_{\eps\goto 0}
\norm{\frac{g^\eps\circ\cT^\eps - g^0}{\eps} - 
\nabla g^0\cdot V}{L^2(U^0)}
=
0.
\label{9Nov1b}
\end{align}
Let $u^\eps$ and $u^0$ be solutions 
to~\eqref{e:sh1} 
and~\eqref{e:sh2}, respectively.
The material derivative of $u^\eps$ exists and is 
the 
solution of the following 
problem: Find $z\in L^2(0,T; H_0^1(U^0)$ 
satisfying
\begin{equation}\label{e:A4}
\begin{aligned}
z_t(t,\vecy) 
-
\Delta z(t,\vecy) 
=
&
-u_t^0(t,\vecy)\,\divv V(\vecy)
-
\nabla \cdot 
\left( 
A'(0,\vecy) \nabla u^0(t,\vecy)
\right)
\\
&
+ 
\divv\Big(V(\vecy)\, f^0(t,\vecy)\Big),
\quad 
(t,\vecy)\in [0,T]\times U^0
\\
z(t,\vecy) 
=
&\
0,
\quad 
(t,\vecy)\in [0,T]\times\G^0
\\
z(0,\vecy) 
=
& \
\nabla g^0(\vecy)\cdot V(\vecy),
\quad 
\vecy\in U^0.
\end{aligned}
\end{equation}
\end{lemma}

\begin{proof}
First, since  
$f^\eps\in L^2(Q_\cT^{\eps})$ and $g^\eps\in 
H_0^2(U^\eps)\subset H_0^1(U^\eps)$ for all 
$\eps\ge 0$, by  
Lemma~\ref{l:unique Dir 2}, 
the solution $u^\eps$ 
of~\eqref{e:weak for} exists and $u^\eps\in 
L^\infty(0,T; H_0^1(U^\eps)) \cap H^1(0,T; 
L^2(U^\eps))$ for any $\eps\ge 0$. In particular, 
when $\eps =0$, the solution 
$u^0$ belongs to 
$L^\infty(0,T; H_0^1(U^0)) \cap H^1(0,T; 
L^2(U^0))$. 
Note also that $V\in W^{1,\infty}(U^0)$.
The unique existence of the solution $z$ 
of~\eqref{e:A4} is 
assured 
by using~\cite[Theorem 11.1.1 and Remark 
11.1.1]{QuaVal97}.
Furthermore, the corresponding weak formulation 
is: 
Find $z\in L^2(0,T; 
H_0^1(U^0))$ satisfying 
\begin{align}
\int_{U^0}
z_t(t,\vecy)\,w(\vecy)\,d\vecy
\
+
&
\int_{U^0}
[\nabla w(\vecy)]^T \,
\nabla z(t,\vecy)
\,d\vecy
=
- 
\int_{U^0}
u_t^0(t,\vecy) \,
w(\vecy)
\,
\divv V(\vecy)\,
d\vecy
\notag
\\
&
-
\int_{U^0}
\nabla \cdot 
\left( 
A'(0,\vecy)\,\nabla u^0(t,\vecy)
\right)
w(\vecy)
\,d\vecy
\notag 
\\
&
+
\int_{U^0}
\divv\left(V(\vecy)\,f^0(t,\vecy)\right)
\,
w(\vecy)\,d\vecy
\quad 
\forall w\in H_0^1(U^0).
\label{e:A5}
\end{align}
Here, the second integral on the right hand side 
can be 
rewritten by using the Divergence Theorem
\begin{align}
\int_{U^0}
\nabla \cdot 
\left( 
A'(0,\vecy)\,\nabla u^0(t,\vecy)
\right)
w(\vecy)
\,d\vecy
&
=
-
\int_{U^0} 
\big[ 
\nabla w(\vecy) 
\big]^T\,
A'(0,\vecy)\,
\nabla u^0(t,\vecy)\,
d\vecy,
\notag
\end{align}
noting that $w$ vanishes on the boundary of $U^0$.
On the other hand, by
dividing both sides of~\eqref{e:u eps u D0} by $\eps$  we 
have 
\begin{align}
&
\int_{U^0}
\frac{
(
u_t^\eps\dot{\circ}\cT^\eps
-
u_t^0)(t,\vecy)
}{\eps}\,
w(\vecy)\,d\vecy
+
\int_{U^0}
[\nabla w(\vecy)]^T \,
\nabla\frac{(u^\eps\dot{\circ}\cT^\eps - 
u^0)(t,\vecy)}{\eps}
\,d\vecy
\notag
\\
=
&
- 
\int_{U^0}
u_t^\eps\dot{\circ}\cT^\eps(t,\vecy) \,
w(\vecy)
\,
\frac{
\gamma(\eps,\vecy) - 1
}{\eps}\,
d\vecy
-
\int_{U^0}
[\nabla w(\vecy)]^T \,
\frac{
A(\eps,\vecy) - I
}{\eps}
\,
\nabla(u^\eps\dot{\circ}\cT^\eps)(t,\vecy)\,d\vecy
\notag 
\\
&
+
\int_{U^0}
\frac{
f^\eps\dcirc\cT^\eps(t,\vecy)
\, \gamma(\eps,\vecy) 
-
f^0(t,\vecy)
}{\eps}\,
w(\vecy)\,d\vecy
\quad 
\forall w\in H_0^1(U^0).
\label{e:u eps u D0 1}
\end{align}
Subtracting~\eqref{e:A5} from~\eqref{e:u eps u D0 1} we 
have 
\begin{align}
\int_{U^0}
&
\Big[
\frac{
(u_t^\eps\dot{\circ}\cT^\eps
-
u_t^0)(t,\vecy)
}{\eps}
-
z_t(t,\vecy)
\Big]
w(\vecy)\,d\vecy
\notag 
\\
&
+
\int_{U^0}
[\nabla w(\vecy)]^T \,
\nabla
\Big[
\frac{(u^\eps\dot{\circ}\cT^\eps - u^0)(t,\vecy)}{\eps}
-
z(t,\vecy)
\Big]
\,d\vecy
\notag
\\
=
&
- 
\int_{U^0}
w(\vecy)
\Big[
u_t^\eps\dot{\circ}\cT^\eps(t,\vecy) \,
\frac{
\gamma(\eps,\vecy) - 1
}{\eps}
-
u_t^0(t,\vecy)\divv V(\vecy)
\Big]
\,
d\vecy
\notag
\\
&
-
\int_{U^0}
[\nabla w(\vecy)]^T 
\Big[
\frac{
A(\eps,\vecy) - I
}{\eps}
\,
\nabla(u^\eps\dot{\circ}\cT^\eps)(t,\vecy)
- 
A'(0,\vecy) \nabla u^0(t,\vecy)
\Big]
\,d\vecy
\notag 
\\
&
+
\int_{U^0}
\Big[
\frac{
f^\eps\dcirc\cT^\eps(t,\vecy)
\, \gamma(\eps,\vecy) 
-
f^0(t,\vecy)
}{\eps}
-
\divv\left(V(\vecy)\,f^0(t,\vecy)\right)
\Big]
\,
w(\vecy)\,d\vecy
\quad 
\forall w\in H_0^1(U^0).
\label{e:A6}
\end{align}
For each $t\in (0,T)$, substituting $w(\vecy)
=
(u^\eps\dot{\circ}\cT^\eps - u^0)(t,\vecy)/\eps - 
z(t,\vecy)$ into~\eqref{e:A6} and then integrating 
both sides over 
$[0,t]$, we obtain 
\begin{align}
&
\frac12
\int_{U^0}
\Big[
\frac{
u^\eps\dot{\circ}\cT^\eps(t,\vecy) 
-
u^0(t,\vecy)
}{\eps}
-
z(t,\vecy)
\Big]^2
d\vecy
+
\int_0^t
\norm{
\nabla
\Big[
\frac{(u^\eps\dot{\circ}\cT^\eps - u^0)(\tau)}{\eps}
-
z(\tau)
\Big]
}{L^2(U^0)}^2
\hspace{-0.4cm}
d\tau
\notag
\\
=
&
\
\frac{1}{2}
\int_{U^0}
\Big[ 
\frac{g^\eps{\circ}\cT^\eps(\vecy) - 
g^0(\vecy)}{\eps} 
-
\nabla g^0(\vecy)\cdot V(\vecy)
\Big]^2\,d\vecy
\notag 
\\
&
- 
\int_0^t
\int_{U^0}
\Big[
\frac{(u^\eps\dot{\circ}\cT^\eps - 
u^0)
(\tau,\vecy)}{\eps}
-
z(\tau,\vecy)
\Big]
\notag
\\
&
\hspace{3cm}
\times
\Big[
u_t^\eps\dcirc\cT^\eps(\tau,\vecy) \,
\frac{
\gamma(\eps,\vecy) - 1
}{\eps}
-
u_t^0(\tau,\vecy)\divv V(\vecy)
\Big]
d\vecy
\,d\tau
\notag
\\
&
-
\int_0^t
\int_{U^0}
\nabla
\Big[
\frac{(u^\eps\dot{\circ}\cT^\eps - u^0)(\tau,\vecy)}{\eps}
-
z(\tau,\vecy)
\Big]^T 
\notag 
\\
&
\hspace{3cm}
\times
\Big[
\frac{
A(\eps,\vecy) - I
}{\eps}
\,
\nabla(u^\eps\dot{\circ}\cT^\eps)(\tau,\vecy)
- 
A'(0,\vecy) \nabla u^0(\tau,\vecy)
\Big]
d\vecy\,d\tau
\notag 
\\
&
+
\int_0^t
\int_{U^0}
\Big[
\frac{
f^\eps\dcirc\cT^\eps(\tau,\vecy)
\, \gamma(\eps,\vecy) 
-
f^0(\tau,\vecy)
}{\eps}
-
\divv\left(V(\vecy)\,f^0(\tau,\vecy)\right)
\Big]
\notag
\\
&
\hspace{3cm}
\times
\Big[
\frac{(u^\eps\dot{\circ}\cT^\eps - u^0)(\tau,\vecy)}{\eps}
-
z(\tau,\vecy)
\Big]
d\vecy\,d\tau
\notag
\\
=:\
&
B_1 + B_2 + B_3 + B_4.
\label{e:A7}
\end{align}
Considering the second integral on the right hand side 
of~\eqref{e:A7}, the 
duality and 
Cauchy inequalities give
\begin{align}
\abs{B_2}
\le
&
\int_0^t
\norm{
\frac{(u^\eps{{\dcirc}} \cT^\eps {-} u^0)(\tau)}{\eps}
{-}
z(\tau)
}{H^1(U^0)}
\norm{ 
u_t^\eps{\dcirc}\cT^\eps(\tau) \,
\frac{
\gamma(\eps,\cdot) - 1
}{\eps}
-
u_t^0(\tau)\divv V
}{H^{-1}(U^0)}
d\tau
\notag
\\
\le 
&
\
\al
\int_0^t
\norm{
\frac{(u^\eps{{\dcirc}} \cT^\eps - 
u^0)(\tau)}{\eps}
-
z(\tau)
}{H^1(U^0)}^2
d\tau
\notag
\\
&
+
\frac{1}{4\al} 
\int_0^t
\norm{ 
u_t^\eps{\dcirc}\cT^\eps(\tau)\,
\frac{
\gamma(\eps,\cdot) {-} 1
}{\eps}
{-}
u_t^0(\tau)\divv V
}{H^{-1}(U^0)}^2
d\tau
\label{e:A8}
\end{align}
for any $\al>0$.
Similarly, for any positive numbers $\beta$ and $\eta$, we 
have
\begin{align}
\abs{B_3}
\le 
&
\
\beta
\int_0^t
\norm{
\nabla
\Big[
\frac{(u^\eps{{\circ}} \cT^\eps - 
u^0)(\tau)}{\eps}
-
z(\tau)
\Big]}{L^2(U^0)}^2
d\tau
\notag
\\
&
+
\frac{1}{4\beta} 
\int_0^t
\norm{ 
\frac{
A(\eps,\cdot) - I
}{\eps}
\,
\nabla(u^\eps{{\circ}} \cT^\eps)(\tau)
- 
A'(0,\cdot) \nabla u^0(\tau)}{L^2(U^0)}^2
d\tau,
\label{e:A9}
\end{align}
and 
\begin{align}
\abs{B_4}
\le 
&
\
\frac{1}{4\eta}
\int_0^t
\norm{  
\frac{
f^\eps{\dcirc}\cT^\eps(\tau) 
\, \gamma(\eps,\cdot) 
-
f^0(\tau)
}{\eps}
-
\divv\left(V f^0(\tau)\right)
}{L^2(U^0)}^2
d\tau
\notag
\\
&
+
\eta 
\int_0^t
\norm{
\frac{(u^\eps{{\circ}} \cT^\eps - 
u^0)(\tau)}{\eps}
-
z(\tau)
}{L^2(U^0)}^2
d\tau.
\label{e:A10}
\end{align}
From~\eqref{e:A7}--\eqref{e:A10} and noting that 
$\norm{v}{L^2(U^0)}\le \norm{v}{H^1(U^0)}$ for 
all $v\in H^1(U^0)$, we obtain 
\begin{align}
&
\frac12
\norm{
\frac{
u^\eps\dot{\circ}\cT^\eps(t)
-
u^0(t)
}{\eps}
-
z(t)
}{L^2(U^0)}^2
+
\int_0^t
\norm{
\nabla
\Big[
\frac{(u^\eps\dot{\circ}\cT^\eps - u^0)(\tau)}{\eps}
-
z(\tau)
\Big]
}{L^2(U^0)}^2
\,d\tau
\notag
\\
\le\ 
&
\frac{1}{2}
\norm{
\frac{g^\eps{{\circ}}\cT^\eps - g^0}{\eps} 
-
\nabla g^0\cdot V
}{L^2(U^0)}^2
+
\al 
\int_0^t
\norm{
\frac{(u^\eps{{\dcirc}} \cT^\eps - u^0)(\tau)}{\eps}
-
z(\tau)
}{H^1(U^0)}^2
\,d\tau
\notag 
\\
&
+
\frac{1}{4\al}
\int_0^t
\norm{
u_t^\eps{\dcirc}\cT^\eps(\tau)
\frac{
\gamma(\eps,\cdot) - 1
}{\eps}
-
u_t^0(\tau)\divv V
}{H^{-1}(U^0)}^2
d\tau
\notag
\\
&
+
\beta
\int_0^t
\norm{
\frac{(u^\eps{{\circ}} \cT^\eps {-} u^0)(\tau)}{\eps}
{-}
z(\tau)
}{H^1(U^0)}^2
d\tau
\notag 
\\
&
+
\frac{1}{4\beta} 
\int_0^t
\norm{ 
\frac{
A(\eps,\cdot) - I
}{\eps}
\,
\nabla(u^\eps{{\dcirc}} \cT^\eps)(\tau)
- 
A'(0,\cdot) \nabla u^0(\tau)}{L^2(U^0)}^2
d\tau
\notag 
\\
&
+
\frac{1}{4\eta}
\int_0^t
\norm{
\frac{
f^\eps{\dcirc}\cT^\eps(\tau)
\, \gamma(\eps,\cdot) 
-
f^0(\tau)
}{\eps}
-
\divv\left(V\,f^0(\tau)\right)
}{L^2(U^0)}^2
d\tau
\notag
\\
&
+
\eta 
\int_0^t
\norm{
\frac{(u^\eps{{\circ}} \cT^\eps - u^0)(\tau)}{\eps}
-
z(\tau)
}{H^1(U^0)}^2
d\tau.
\label{e:A11}
\end{align}
The Poinc\'are Inequality gives 
\begin{align}
\sqrt{C_1}
\norm{
\frac{(u^\eps\dot{\circ}\cT^\eps - u^0)(\tau)}{\eps}
-
z(\tau)
}{H^1(U^0)}
&
\le 
\norm{ 
\nabla 
\Big[
\frac{(u^\eps\dot{\circ}\cT^\eps - u^0)(\tau)}{\eps}
-
z(\tau)
\Big]
}{L^2(U^0)}
\label{e:A12}
\end{align}
for some $C_1\in (0,1)$. This together with~\eqref{e:A11} 
implies
\begin{align}
&
\frac12
\norm{
\frac{
u^\eps\dot{\circ}\cT^\eps(t)
-
u^0(t)
}{\eps}
-
z(t)
}{L^2(U^0)}^2
+
\left( 
{C_1} 
{-}
\al
{-}
\beta
{-}
\eta
\right)
\int_0^t
\norm{
\frac{(u^\eps\dot{\circ}\cT^\eps {-} u^0)(\tau)}{\eps}
{-}
z(\tau)
}{H^1(U^0)}^2
\,d\tau
\notag
\\
\le
\
&
\frac{1}{2}
\norm{
\frac{g^\eps{{\circ}}\cT^\eps - g^0}{\eps} 
-
\nabla g^0\cdot V
}{L^2(U^0)}^2
+
\frac{1}{4\al} 
\int_0^t
\norm{
u_t^\eps{\dcirc}\cT^\eps(\tau)
\frac{
\gamma(\eps,\cdot) - 1
}{\eps}
-
u_t^0(\tau)\divv V
}{H^{-1}(U^0)}^2
d\tau
\notag 
\\
+
\
&
\frac{1}{4\beta} 
\int_0^t
\norm{ 
\frac{
A(\eps,\cdot) - I
}{\eps}
\,
\nabla(u^\eps{{\dcirc}} \cT^\eps)(\tau)
- 
A'(0,\cdot) \nabla u^0(\tau)}{L^2(U^0)}^2
d\tau
\notag 
\\
+
\
&
\frac{1}{4\eta}
\int_0^t
\norm{
\frac{
f^\eps{\dcirc}\cT^\eps(\tau)
\, \gamma(\eps,\cdot) 
-
f^0(\tau)
}{\eps}
-
\divv\left(V f^0(\tau)\right)
}{L^2(U^0)}^2
d\tau
\label{e:A13}
\end{align}
for all positive numbers $\al, \beta$ and $\eta$.
We can choose $\al = \beta = \eta = 
{C_1}/{4}$ to 
obtain 
\begin{align}
&
\norm{
\frac{
u^\eps\dot{\circ}\cT^\eps(t)
-
u^0(t)
}{\eps}
-
z(t)
}{L^2(U^0)}^2
+
\int_0^t
\norm{
\frac{(u^\eps\dot{\circ}\cT^\eps - u^0)(\tau)}{\eps}
-
z(\tau)
}{H^1(U^0)}^2
\,d\tau
\notag
\\
\lesssim 
\
&
\norm{
\frac{g^\eps{{\circ}}\cT^\eps - g^0}{\eps} 
-
\nabla g^0\cdot V
}{L^2(U^0)}^2
+
\int_0^t
\norm{
u_t^\eps{\dcirc}\cT^\eps(\tau)
\frac{
\gamma(\eps,\cdot) - 1
}{\eps}
-
u_t^0(\tau)\divv V
}{H^{-1}(U^0)}^2
d\tau
\notag 
\\
+
\
&
\int_0^t
\norm{ 
\frac{
A(\eps,\cdot) - I
}{\eps}
\,
\nabla(u^\eps{{\dcirc}} \cT^\eps)(\tau)
- 
A'(0,\cdot) \nabla u^0(\tau)}{L^2(U^0)}^2
d\tau
\notag 
\\
+
\
&
\int_0^t
\norm{
\frac{
f^\eps{\dcirc} \cT^\eps(\tau)
\, \gamma(\eps,\cdot) 
-
f^0(\tau)
}{\eps}
-
\divv\left(V f^0(\tau)\right)
}{L^2(U^0)}^2
d\tau.
\label{e:A14}
\end{align}
Letting $\eps\goto 0$ and noting~\eqref{e:T g 2} and Lemma~\ref{l:9Au1a}, we 
obtain 
\begin{align}
\lim_{\eps\goto 0}
\norm{
\frac{
u^\eps\dot{\circ}\cT^\eps
-
u^0
}{\eps}
-
z
}{L^\infty(0,T;L^2(U^0))}
&
=
0,
\notag
\end{align}
and 
\begin{align}
\lim_{\eps\goto 0}
\norm{
\frac{u^\eps\dot{\circ}\cT^\eps - u^0}{\eps}
-
z
}{L^2(0,T;H^1(U^0))}
&
=
0.
\label{9Au1h}
\end{align}
Noting~\eqref{e:A6}
and employing the duality argument, 
we 
obtain
\begin{align}
&
\norm{ 
\frac{u_t^\eps{\circ}\cT^\eps - u_t^0}{\eps}
-
z_t
}{L^2(0,T;H^{-1}(U^0))}
\\
\lesssim 
&
\norm{ 
\frac{u^\eps\dot{\circ}\cT^\eps - u^0}{\eps}
-
z
}{L^2(0,T;H^1(U^0))}
+
\norm{ 
u_t^\eps{\circ}\cT^\eps
\frac{\gamma(\eps,\cdot)-1}{\eps}
-
u_t^0\divv V
}{L^2(0,T;H^{-1}(U^0))}
\notag 
\\
&
+
\norm{
\frac{
A(\eps,\cdot) - I
}{\eps}
\,
\nabla(u^\eps\dot{\circ}\cT^\eps)
- 
A'(0,\cdot) \nabla u^0
}{L^2(0,T;L^2(U^0))}
\notag
\\
&
+
\norm{ 
\frac{
f{\circ}\cT^\eps
\, \gamma(\eps,\cdot) 
-
f
}{\eps}
-
\divv\left(V\,f\right)
}{L^2(0,T;L^2(U^0))}.
\notag
\end{align}
This together with~\eqref{9Au1h} and the results in Lemma~\ref{l:9Au1a}
yields
\begin{align}
\lim_{\eps\goto 0}
\norm{ 
\frac{u_t^\eps{\circ}\cT^\eps - u_t^0}{\eps}
-
z_t
}{L^2(0,T;H^{-1}(U^0))}
=
0,
\end{align}
finishing the proof of this lemma.
\end{proof}

Lemma~\ref{l:material der} assures the existence of the 
material derivative of the solution $u^\eps$ of the 
perturbed problem~\eqref{e:sh1} and thus assures the 
existence of the corresponding shape derivative. 
Furthermore, the shape derivative turns out to be the 
solution of a heat equation.

\begin{theorem}
Under the assumptions of Lemma~\ref{l:material der}, the 
shape derivative $u'$ of $u^\eps$ exists and is the weak 
solution of the following heat problem
\begin{subequations}\label{e:B14}
\begin{align}
u'_t(t,\vecx) 
-
\Delta u'(t,\vecx) 
&
=
0,
\quad 
(t,\vecx)\in (0,T)\times U^0
\label{e:B14 11}
\\
u'(t,\vecx) 
&
=
-
\nabla
u^0(t,\vecx)\cdot V(\vecx),
\quad 
(t,\vecx)\in (0,T)\times\G^0
\label{e:B14 21}
\\
u'(0,\vecx) 
&
=
0,
\quad 
\vecx\in U^0.
\label{e:B14 31}
\end{align}
\end{subequations}

\end{theorem}
 
\begin{proof}
Let $v$ be an arbitrary function in $C_0^\infty(U^0)$. Then 
there is a sufficiently small $\eps_0>0$ such that for all 
$\eps<\eps_0$, $v\in C_0^\infty(U^\eps)$. Multiplying both 
sides of the first equation in~\eqref{e:sh1} with 
$v$ and then integrating over $U^\eps$, using the Green's 
identity, we obtain 
\begin{align}
\int_{U^\eps} u_t^\eps(t,\vecx)\, v(\vecx)\,d\vecx
-
\int_{U^\eps} 
u^\eps(t,\vecx)\,
\Delta v(\vecx)\,d\vecx
&
=
\int_{U^\eps}
f(t,\vecx)\,v(\vecx)\,d\vecx.
\label{e:B8}
\end{align}
Integrating both sides of~\eqref{e:B8} over $[0,t]$, we 
obtain 
\begin{align}
\int_{U^\eps} 
[u^\eps(t,\vecx) - g(\vecx)]\,v(\vecx)
\,d\vecx
-
\int_0^t 
\int_{U^\eps}
u^\eps(\tau,\vecx)\,\Delta v(\vecx)\,d\vecx\,d\tau
&
=
\int_0^t 
\int_{U^\eps}
f(\tau,\vecx)\,v(\vecx)\,d\vecx\,d\tau,
\notag
\end{align}
noting~\eqref{e:B14 21}.
We then differentiate both sides of the above equation over 
$\eps$ to obtain 
\begin{align}
&
\int_{U^0}
u'(t,\vecx)\,v(\vecx)\,d\vecx
+
\int_{\G^0}
[u^0(t,\vecx)-g(\vecx)]\,v(\vecx)
\inpro{V(\vecx)}
{\vecn^0(\vecx)}\,d\sigma_{\vecx}
\notag 
\\
&
-
\int_0^t 
\int_{U^0}
u'(\tau,\vecx)\,\Delta v(\vecx)\,d\vecx\,d\tau
-
\int_0^t 
\int_{\G^0}
u^0(\tau,\vecx)\,\Delta v(\vecx)
\inpro{V(\vecx)}
{\vecn^0(\vecx)}\,d\sigma_{\vecx}\,d\tau
\notag
\\
=
&
\int_0^t 
\int_{\G^0}
f(\tau,\vecx)\,v(\vecx)
\inpro{V(\vecx)}
{\vecn^0(\vecx)}\,d\sigma_{\vecx}\,d\tau.
\notag
\end{align}
Noting that $v\in C_0^\infty(U^0)$, there holds
\begin{align}
\int_{U^0}
u'(t,\vecx)\,v(\vecx)\,d\vecx
-
\int_0^t 
\int_{U^0}
u'(\tau,\vecx)\,\Delta v(\vecx)\,d\vecx\,d\tau
&
=
0
\quad 
\forall v\in C_0^\infty(U^0).
\label{e:B9}
\end{align}
Rewriting the first integral in form of integral over 
$[0,t]$ and 
applying the Green's identity again for the second integral 
we arrive at 
\begin{align}
\int_0^t 
\int_{U^0} 
u'_t(\tau,\vecx)\,v(\vecx)\,d\vecx\,d\tau
-
\int_0^t 
\int_{U^0}
\Delta 
u'(\tau,\vecx)\,v(\vecx)\,d\vecx\,d\tau
+
\int_{U^0}
u'(0,\vecx)\,v(\vecx)
=
0.\label{e:B10}
\end{align}
It follows from the definition of shape derivatives, 
equations~\eqref{e:A4} and~\eqref{e:sh2c} that 
\begin{align} 
u'(0,\vecx)
&
= \dot{u}(0,\vecx) 
- 
\nabla 
u^0(0,\vecx)\cdot V(\vecx) 
\notag 
\\
&
=
\nabla g(\vecx)\cdot V(\vecx) 
-
\nabla 
u^0(0,\vecx)\cdot V(\vecx) 
=
0
\quad 
\forall \vecx\in U^0.
\label{e:B12}
\end{align}
This together with~\eqref{e:B10} yields 
\begin{align}
\int_0^t 
\int_{U^0} 
u'_t(\tau,\vecx)\,v(\vecx)\,d\vecx\,d\tau
-
\int_0^t 
\int_{U^0}
\Delta 
u'(\tau,\vecx)\,v(\vecx)\,d\vecx\,d\tau
=
0
\quad 
\forall 
v\in H_0^1(U^0).
\label{e:B11}
\end{align}
The condition on boundary surface of the material 
derivative (see~\eqref{e:A4} and the definition of shape 
derivative) gives 
\begin{align}
u'(t,\vecx) 
= 
-\nabla u^0(t,\vecx)\cdot V(\vecx),
\quad 
(t,\vecx)\in (0,T)\times\G^0.
\label{e:B13}
\end{align}
From~\eqref{e:B12}--\eqref{e:B13}, the shape derivative is 
the solution to the following problem
\begin{subequations}\label{e:br1}
\begin{align}
u'_t(t,\vecx) 
-
\Delta u'(t,\vecx) 
&
=
0,
\quad 
(t,\vecx)\in U^0\times (0,T)
\label{e:B14 1}
\\
u'(t,\vecx) 
&
=
-
\nabla
u^0(t,\vecx)\cdot V(\vecx),
\quad 
(t,\vecx)\in (0,T)\times\G^0
\label{e:B14 2}
\\
u'(0,\vecx) 
&
=
0,
\quad 
\vecx\in U^0,
\label{e:B14 3}
\end{align}
\end{subequations}
finishing the proof of the theorem.
\end{proof}

\section{Boundary reduction} 


Let $r$ and $s$ be nonnegative real numbers. For 
$U$ an open set in $\R^n$ and for $0< 
T<\infty$, we denote 
$Q_T := (0,T)\times U$ and $\Sigma_T := 
(0,T)\times \G$ where $\Sigma$ is the boundary 
of $U$. The space $H^{r,s}(Q_T)$ is 
defined by 
\begin{equation}\label{HrsO1}
H^{r,s}(Q_T)
=
L^2(0,T; H^r(U))
\cap 
H^s(0,T; L^2(U)).
\end{equation}
This is a Hilbert space with the following norm 
\[
\norm{v}{H^{r,s}(Q_T)}
=
\left( 
\int_0^T 
\norm{v(t)}{H^r(U)}^2\,dt 
+
\norm{v}{H^s(0,T; L^2(U))}^2
\right)^{1/2}.
\]
The space $H^{r,s}(\Sigma_T)$ is analogously 
defined with the corresponding norm 
\[
\norm{v}{H^{r,s}(\Sigma_T)}
=
\left( 
\int_0^T 
\norm{v(t)}{H^r(\G)}^2\,dt 
+
\norm{v}{H^s(0,T; L^2(\G))}^2
\right)^{1/2}.
\]
The following property whose proof can be found 
in~\cite[Proposition~2.1]{LioMagII} states the 
interpolation property of the time-varying 
Sobolev spaces $H^{r,s}(Q_T)$ and 
$H^{r,s}(\Sigma_T)$.

\begin{proposition}
For $r_1, r_2, s_1, s_2\ge 0$ and $\theta\in 
[0,1]$, there hold 
\begin{align}
\left[ 
H^{r_1,s_1}(Q_T),
H^{r_2,s_2}(Q_T)
\right]_{\theta}
=
H^{\theta r_1+(1-\theta)r_2,\ \theta 
s_1+(1-\theta)s_2}(Q_T),
\notag 
\\
\left[ 
H^{r_1,s_1}(\Sigma_T),
H^{r_2,s_2}(\Sigma_T)
\right]_{\theta}
=
H^{\theta r_1+(1-\theta)r_2,\ \theta 
s_1+(1-\theta)s_2}(\Sigma_T).
\end{align}

\end{proposition}

The space $H^{r,s}(Q_T)$ turns out to be  the 
space of restrictions to $Q_T$ of the functions 
in $H^{r,s}(\R;U)$, equipped with the obvious 
quotient norm. Here,  
the space $H^{r,s}(\R; U)$ is defined by 
\[
H^{r,s}(\R;U) 
:=
L^2(\R; H^r(U)) \cap H^s(\R; L^2(U))
\]
with the natural norm defined on these spaces of 
Hilbert space valued distributions. If 
\[
\wtd v(\tau,\vx) 
:=
\frac{1}{\sqrt{2\pi}}
\int_{\R}
e^{-i t \tau}
\,
v(t,\vx)\,dt
\]
is the Fourier transform of $u$ w.r.t. the time 
variable, we have 
\[
\norm{v}{H^{r,s}(\R;\R^n)}^2
=
\int_{\R} 
\brap{ 
\norm{\wtd v(\tau;\cdot)}{H^r(U)}^2
+
\brac{1+\abs{\tau}^2}^s
\norm{\wtd v(\tau)}{L^2(U)}^2
}
\,d\tau.
\]
For $r,s\le0$, the space $H^{r,s}(Q_T)$, 
$H^{r,s}(\Sigma_T)$ and $H^{r,s}(\R;U)$ are 
defined to be the dualities of 
$H^{-r,-s}(Q_T)$, 
$H^{-r,-s}(\Sigma_T)$ and $H^{-r,-s}(\R;U)$, 
respectively. We also need the following subspaces
\[
\wtd H^{r,s}(Q_T) 
:=
\left\{ 
v\in H^{r,s}((-\infty, T);U)\ |\ u(t,\vx) = 
0\quad \text{for}\quad t<0
\right\}
\]
and 
\begin{align}
\cV(Q_T)
=
L^2((0,T); H^1(U))\cap H^1((0,T); H^{-1}(U)).
\notag
\end{align}
In this section, boundary integral equation 
methods will then be used to compute 
statistical moments of the shape derivative 
(see~\eqref{e:br1}). 
We recall here some required boundary integral 
operators.

Let $v\in \tilde H^{1,1/2}(\R_+; U)$ be given 
and 
$t_0\in \R$ be arbitrary. Define the time-reversal map 
$\lambda_{t_0}$ by 
\[
\lambda_{t_0}
v(t,\vecx)
:=
v(t_0-t,\vecx).
\]
Then $\lambda_{t_0}v\in 
H_{t_0}^{1,1/2}\brac{(-\infty,t_0); U}$. 
Let 
\[
G(t,\vecx)
:=
(4\pi t)^{-3/2}
e^{-\frac{\abs{\vecx}^2}{4t}}\vartheta(t)
\]
be the fundamental solution of the heat equation, where 
$\vartheta(t) = 1/2(1+\textrm{sign}t)$ is the Heaviside 
function.
Denoting $\mathcal{G}(t,\vecx) = G(t,\vecx_0-\vecx)$, then 
$\cG\in\tilde H^{1,1/2}(\R_+; U)$ and 
$\vartheta_t(t,\vecx) -\Delta\vartheta(t,\vecx) = 0$ for 
$\vecx\not=\vecx_0$.

For $(t_0,\vecx_0)\in Q_T:=(0,T)\times U$, the single layer 
potential is defined by 
\begin{align}
K_0(v)(t_0,\vecx_0)
&
:=
{\int_0^T}
{\int_U} 
v\,
\gamma_0\brac{{\lambda_{t_0}\cG}}
d\vecx\,dt
\notag
\\
&
=
{\int_0^T}
{\int_U} 
v(t,\vecx)
\gamma_0
\brac{ 
[4\pi (t_0-t)]^{-3/2}
e^{-\frac{\abs{\vecx-\vecx_0}^2}{4(t_0-t)}
}
\vartheta(t_0-t)
}\,d\vecx\,dt,
\label{e:sing potential}
\end{align}
and the double layer potential is defined by 
\begin{align}
K_1(v)(t_0,\vecx_0)
&
:=
\int_0^T\int_U 
v\,
\gamma_1\brac{\lambda_{t_0}\cG}\,d\vecx\,dt
\notag
\\
&
=
\int_0^T\int_U 
v(t,\vecx)\,
\gamma_1
\brac{ 
[4\pi (t_0-t)]^{-3/2}
e^{-\frac{\abs{\vecx-\vecx_0}^2}{4(t_0-t)}
}
\vartheta(t_0-t)
}d\vecx\,dt.
\label{e:double potential}
\end{align}
The boundary integral operators are defined as 
follows:
\begin{align*}
(\cV\psi )(t,\vx)
&
:=
\lim_{U^0\ni\wtd\vx\goto \vx}
(K_0\psi)(t,\wtd\vx),
\quad 
\vx\in \G^0,
\\
(\cN\psi)(t,\vx)
&
:=
\frac{1}{2}
\brac{ 
\lim_{U^0\ni \wtd\vx\goto\vx}
\nabla_{\wtd\vx}
\brac{K_0\psi}\cdot\vn_{\vx}
+
\lim_{(U^0)^c\ni \wtd\vx\goto\vx}
\nabla_{\wtd\vx}
\brac{K_0\psi}\cdot\vn_{\vx}
},
\quad 
\vx\in \G^0,
\\
(\cK w) 
&
:=
\frac{1}{2}
\brac{ 
\lim_{U^0\ni \wtd\vx\goto\vx}
\brac{K_1\psi}(t,\vx)
+
\lim_{(U^0)^c\ni \wtd\vx\goto\vx}
\brac{K_1\psi}(t,\vx)
},
\quad 
\vx\in \G^0,
\\
(\cW w)(t,\vx) 
&
:=
-
\lim_{U^0\ni \wtd\vx\goto\vx}
(K_1 w)(t,\vx), 
\quad 
\vx\in \G^0,
\end{align*}
for 
$\psi\in H^{-1/2,-1/4}((0,T); \G^0)$ and 
$w\in H^{1/2,1/4}((0,T); \G^0)$.

The unique solution of~\eqref{e:br1} can be represented 
\begin{enumerate}[(a)]
\item 
as $u' = K_0 \psi - K_1(-\nabla u^0\cdot V))$, where 
$\psi$ is the unique solution of the first kind integral 
equation 
\begin{equation}\label{e:br3}
\cV\psi 
=
\brac{\frac12 I + \cK}(-\nabla u^0\cdot V)).
\end{equation}
\item 
as $u' = K_0 \psi - K_1(-\nabla_{\G^0} u^0\cdot V))$, where 
$\psi$ is the unique solution of the second kind integral 
equation 
\begin{equation}\label{e:br2}
\brac{\frac12 I - \cN}\psi 
=
\cW(-\nabla u^0\cdot V)).
\end{equation}
\item 
as $u' = K_0\psi$, where $\psi$ is the unique 
solution of the first kind integral equation
\begin{equation}\label{br3}
\cV\psi = -\nabla u^0\cdot V.
\end{equation}
\item 
as $u = K_1 w$, where $w$ is the unique solution 
of the second kind integral equation 
\begin{equation}\label{br4}
\brac{\frac12 I - \cK} w = \nabla u^0\cdot V.
\end{equation}

\end{enumerate}

We may use the 
representations~\eqref{e:br3}--\eqref{br4} to 
compute the 
statistical moments of the shape 
derivative~\eqref{e:B14}.
As a model, we shall present the boundary 
integral equation method to compute statistical 
moments of the solution to the 
problem~\eqref{e:initial rand prob} by 
using~\eqref{br3}.  Taking the randomness of the 
domain into account. The random shape derivative 
of the solutions to~\eqref{e:initial rand prob} 
satisfies the following problem 
\begin{subequations}\label{brRad1}
\begin{align}
u'_t(t,\vecx;\om) 
-
\Delta u'(t,\vecx;\om) 
&
=
0,
\quad 
(t,\vecx)\in U^0\times (0,T)
\label{B14 1}
\\
u'(t,\vecx;\om) 
&
=
-
\nabla
u^0(t,\vecx)\cdot V(\vecx;\om),
\quad 
(t,\vecx)\in (0,T)\times\G^0
\label{B14 2}
\\
u'(0,\vecx;\om) 
&
=
0,
\quad 
\vecx\in U^0.
\label{B14 3}
\end{align}
\end{subequations}
Following~\eqref{br3}, 
the solution $u'(t,\vx;\om)$
can be represented as 
\begin{equation}\label{brRad2}
u'(t,\vx;\om)
=
\brac{K_0 \psi(\om)}(t,\vx), 
\quad 
(t,\vx)\in (0,T)\times U^0,
\end{equation}
where $\psi(\om)$ is the solution of 
\begin{equation}\label{brRad3}
\cV \psi(\om) 
=
-\nabla u^0\cdot V(\cdot;\om).
\end{equation}
Tensorising and integrating~\eqref{brRad2} 
yield 
\begin{equation}\label{brRad4}
\cM^k[u']
=
(K_0)^{(k)}\cM^k[\psi]
\
\text{in }
\bigotimes_{1}^k
\wtd H^{1,1/2}(Q_T^0; \partial_t - \De),
\end{equation}
where $\cM^k[\psi]\in \bigotimes_1^k 
H^{-1/2,-1/4}(\Sigma_T^0)$ is the solution of the 
following equation
\begin{equation}\label{brRad5}
\cV^{(k)} \cM^k[\psi] 
=
(-1)^k 
\cM^k[\kappa]
\bigotimes_1^k 
\brac{\nabla u^0\cdot \vn^0}
\
\text{in}
\
\bigotimes_1^k 
H^{1/2,1/4}(\Sigma_T^0).
\end{equation}



\begin{thebibliography}{10}

\bibitem{BarSwbZol11}
A.~Barth, C.~Schwab and N.~Zollinger.
\newblock Multi-level {M}onte {C}arlo finite element method for elliptic {PDE}s
  with stochastic coefficients.
\newblock {\em Numer. Math.},  {\bf 119} (2011), 123--161.

\bibitem{CohDVeSwb11}
A.~Cohen, R.~Devore and C.~Schwab.
\newblock Analytic regularity and polynomial approximation of parametric and
  stochastic elliptic {PDE}'s.
\newblock {\em Anal. Appl. (Singap.)},  {\bf 9} (2011), 11--47.

\bibitem{Aubin00}
J.-P. Aubin.
\newblock {\em Applied functional analysis}.
\newblock Pure and Applied Mathematics (New York). Wiley-Interscience, New
  York, second edition, 2000.
\newblock With exercises by Bernard Cornet and Jean-Michel Lasry, Translated
  from the French by Carole Labrousse.

\bibitem{ChePhaTra15}
A.~Chernov, T.~D. Pham, and T.~Tran.
\newblock A shape calculus based method for a transmission problem with random
  interface.
\newblock {\em Computers and Mathematics with Applications},  {\bf } (2015).

\bibitem{ChSwb13fokm}
A.~Chernov and C.~Schwab.
\newblock First order {$k$}-th moment finite element analysis of nonlinear
  operator equations with stochastic data.
\newblock {\em Math. Comp.},  {\bf 82} (2013), 1859--1888.

\bibitem{CohDVeSwb10}
A.~Cohen, R.~DeVore, and C.~Schwab.
\newblock Convergence rates of best {$N$}-term {G}alerkin approximations for a
  class of elliptic s{PDE}s.
\newblock {\em Found. Comput. Math.},  {\bf 10} (2010), 615--646.

\bibitem{Evans10}
L.~C.~Evans.
\newblock {\em Partial Differential Equations: Second Edition (Graduate Studies
  in Mathematics)}.
\newblock American Mathematical Society, United States of America, 2010.

\bibitem{ForKor10}
R.~Forster and R.~Kornhuber.
\newblock A polynomial chaos approach to stochastic variational inequalities.
\newblock {\em J. Numer. Math.},  {\bf 18} (2010), 235--255.

\bibitem{Git13adap}
C.~J. Gittelson.
\newblock An adaptive stochastic {G}alerkin method for random elliptic
  operators.
\newblock {\em Math. Comp.},  {\bf 82} (2013), 1515--1541.

\bibitem{HrbSndSwb08sm}
H.~Harbrecht, R.~Schneider and C.~Schwab.
\newblock Sparse second moment analysis for elliptic problems in stochastic
  domains.
\newblock {\em Numer. Math.},  {\bf 109} (2008), 385--414.

\bibitem{BNobTmp07}
I.~Babu\v{s}ka, F.~Nobile and R.~Tempone.
\newblock A stochastic collocation method for elliptic partial differential
  equations with random input data.
\newblock {\em SIAM J. Numer. Anal.},  {\bf 45} (2007), 1005--1034.

\bibitem{GraKuoNuySclSlo11}
I.~G.~Graham, F.~Y.~Kuo, D.~Nuyens, R.~Scheichl and I.~H. Sloan.
\newblock Quasi-{M}onte {C}arlo methods for elliptic {PDE}s with random
  coefficients and applications.
\newblock {\em J. Comput. Phys.},  {\bf 230} (2011), 3668--3694.

\bibitem{LightCheney80}
W.~A. Light and E.~W. Cheney.
\newblock {\em Approximation theory in tensor product spaces}, volume 1169 of
  {\em Lecture Notes in Mathematics}.
\newblock Springer-Verlag, Berlin, 1985.

\bibitem{LioMagII}
J.~L. Lions and E.~Magenes.
\newblock {\em Non-Homogeneous Boundary Value Problems and Applications II}.
\newblock Springer-Verlag, New York, 1972.

\bibitem{QuaVal97}
A.~Quarteroni and A.~Valli.
\newblock {\em Numerical Approximation of Partial Differential Equations}.
\newblock Springer, Berlin, 1997.

\bibitem{SwbGit11}
C.~Schwab and C.~J. Gittelson.
\newblock Sparse tensor discretizations of high-dimensional parametric and
  stochastic {PDE}s.
\newblock {\em Acta Numer.},  {\bf 20} (2011), 291--467.

\bibitem{SwbTod06}
C.~Schwab and R.~A. Todor.
\newblock Karhunen-{L}o{\`e}ve approximation of random fields by generalized
  fast multipole methods.
\newblock {\em J. Comput. Phys.},  {\bf 217} (2006), 100--122.

\bibitem{SokZol92}
J.~Soko{\l}owski and J.-P. Zol{\'e}sio.
\newblock {\em Introduction to shape optimization}, volume~16 of {\em Springer
  Series in Computational Mathematics}.
\newblock Springer-Verlag, Berlin, 1992.
\newblock Shape sensitivity analysis.

\bibitem{TosWid05}
A.~Toselli and O.~Widlund.
\newblock {\em Domain Decomposition Methods---Algorithms and Theory}, volume~34
  of {\em Springer Series in Computational Mathematics}.
\newblock Springer-Verlag, Berlin, 2005.

\end{thebibliography}

\end{document}